\documentclass[11pt,reqno]{amsart}

\usepackage{amscd,amsmath,amssymb,verbatim}
\usepackage[colorlinks]{hyperref}

\newcommand{\C}{\mathbb{C}}
\newcommand{\R}{\mathbb{R}}
\newcommand{\N}{\mathbb{N}}

\newcommand{\ddbar}{i\partial\overline{\partial}}

\renewcommand{\epsilon}{\varepsilon}
\renewcommand{\leq}{\leqslant}
\renewcommand{\geq}{\geqslant}

\newtheorem{theorem}{Theorem}[section]
\newtheorem*{mtheorem}{Main Theorem}
\newtheorem{lemma}[theorem]{Lemma}
\newtheorem{corollary}[theorem]{Corollary}
\newtheorem{proposition}[theorem]{Proposition}
\newtheorem{reduction}[theorem]{Reduction}

\theoremstyle{definition}
\newtheorem{definition}[theorem]{Definition}
\newtheorem{example}[theorem]{Example}

\numberwithin{equation}{section}

\begin{document}

\title{Asymptotics of K\"ahler-Einstein metrics\\on complex hyperbolic cusps}

\author{Xin Fu}
\address{Department of Mathematics, University of California, Irvine, CA 92697, USA}
\email{fux6@uci.edu}

\author{Hans-Joachim Hein}
\address{Mathematisches Institut, WWU M\"unster, 48149 M\"unster, Germany\newline\hspace*{9pt}
Department of Mathematics, Fordham University, Bronx, NY 10458, USA}
\email{hhein@uni-muenster.de}

\author{Xumin Jiang}
\address{Department of Mathematics, Fordham University, New York, NY 10023, USA}
\email{xjiang77@fordham.edu}

\date{\today}

\begin{abstract}
Let $L$ be a negative holomorphic line bundle over an $(n-1)$-dimensional complex torus $D$. Let $h$ be a Hermitian metric on $L$ such that the curvature form of the dual Hermitian metric defines a flat K\"ahler metric on $D$. Then $h$ is unique up to scaling, and, for some closed tubular neighborhood $V$ of the zero section $D \subset L$, the form $\omega_h = -(n+1)i\partial\overline\partial \log(-{\log h})$ defines a complete K\"ahler-Einstein metric on $V \setminus D$ with ${\rm Ric}(\omega_h) = -\omega_h$. In fact, $\omega_h$ is complex hyperbolic, i.e., the holomorphic sectional curvature of $\omega_h$ is constant, and $\omega_h$ has the usual doubly-warped cusp structure familiar from complex hyperbolic geometry. In this paper, we prove that if $U$ is another closed tubular neighborhood of the zero section and if $\omega$ is a complete K\"ahler-Einstein metric with ${\rm Ric}(\omega) = -\omega$ on $U \setminus D$, then there exist a Hermitian metric $h$ as above and a $\delta \in \mathbb{R}^+$ such that $\omega - \omega_{h} = {O}(e^{-\delta\sqrt{-{\log h}}})$ to all orders with respect to $\omega_h$ as $h \to 0$. This rate is doubly exponential in the distance from a fixed point, and is sharp.
\end{abstract}

\maketitle

\markboth{Asymptotics of K\"ahler-Einstein metrics on complex hyperbolic cusps}{Xin Fu, Hans-Joachim Hein, and Xumin Jiang}

\section{Introduction}

\subsection*{Background and main result}

Complex hyperbolic $n$-space $\C H^n$ is the model space uniformizing K\"ahler manifolds of complex dimension $n\geq 2$ with constant negative holomorphic sectional curvature. $\C H^n$ is most conveniently realized as the unit ball in $\mathbb{C}^n$ together with the Bergman metric, and is one of the first examples of a K\"ahler-Einstein manifold \cite{Kahler,Tian}. Finite-volume quotients of $\mathbb{C}H^n$ are either compact or have cuspidal ends. See \cite{Biq,Mok} for good introductions to complex hyperbolic cusps. In this paper, we are interested in a stability property of these cusps as K\"ahler-Einstein manifolds.

We begin by explaining an alternative description of complex hyperbolic cusps in terms of the Calabi ansatz \cite{dBSp,FangFu}. Let $D$ be a (necessarily projective) complex torus of complex dimension $n-1$ together with a negative holomorphic line bundle $L \to D$. Let $h$ be any Hermitian metric on $L$ such that the curvature form of the Hermitian metric dual to $h$ is a flat K\"ahler metric on $D$. Then $h$ is unique up to scaling, i.e., up to replacing $h$ by $\lambda h$ for some $\lambda \in \R^+$. This scaling freedom can also be understood as a consequence of the natural action of $\C^*$ on the total space of $L$. Viewing $h$ as a smooth function on the total space which is homogeneous of degree $2$ along every fiber, we define $\sigma = -{\log h}$ and $\omega_h$ $=$ $-(n+1)i\partial\overline{\partial} \log \sigma$. Then $\omega_h$ is a K\"ahler form on $\{0 < h < 1\}$ with ${\rm Ric}(\omega_h) = -\omega_h$. If $V$ is any closed tubular neighborhood of the zero section $D \subset L$ of the form $V = \{h \leq \delta\}$ for some $\delta \in (0,1)$, then the Riemannian manifold-with-boundary $(V\setminus D, \omega_h)$ is complete of finite volume. In fact, it is a complex hyperbolic cusp, and any complex hyperbolic cusp arises in this way up to a finite cover. Define $x = \frac{1}{\sigma}$ and $\rho = -\sqrt{(n+1)/2}\log x$. Then $\rho$ is a Busemann function (its gradient has constant length $1$). Also, on the level set $Y_\epsilon = \{x = \epsilon^2\}$, the metric tensor is proportional to $\epsilon^4$ in the $S^1$-directions and to $\epsilon^2$ in the $D$-directions. Here we use the Chern connection of $h$ to define a complement to the $S^1$-directions inside the tangent bundle $TY_\epsilon$, and this decomposition is orthogonal with respect to $\omega_h$.

Write $Y = \partial V$. Datar-Fu-Song \cite{DFS} have proved that for any smooth function $\beta: Y \to \R$ there exists a unique complete K\"ahler-Einstein metric $\omega = \omega_h + i\partial\overline{\partial} u$ on $V$ such that $\omega^n = e^u \omega_h^n$ on $V$, hence in particular ${\rm Ric}(\omega) = -\omega$ on $V$, and $u = \beta$ on $Y$. Moreover, $u$ and all of its covariant derivatives with respect to $\omega_h$ decay like $O((-{\log x})^{-1})$, i.e., linearly in terms of the distance function $\rho$. This improves a more global construction of \cite{CY2,Koba,TY} on projective varieties with cusp singularities.

In this paper, we determine the precise asymptotics of the Datar-Fu-Song solutions \cite{DFS}.

\begin{mtheorem}
Fix a Hermitian metric $h$ on $L$ as above. Let $U$ be a closed tubular neighborhood of the zero section $D \subset L$. Let $\omega$ be a complete K\"ahler-Einstein metric on $U \setminus D$ such that ${\rm Ric}(\omega) = -\omega$. Let $u = \log(\omega^n/\omega_h^n)$, so that $\omega = \omega_h + i\partial\overline{\partial} u$ and, by \cite{DFS}, $u = O((-{\log x})^{-1})$ as $x \to 0^+$ to all orders with respect to $\omega_h$. Then there exists a constant $c \in \R$ such that for all $k \in \N_0$,
\begin{align}\label{eq-mainthm}
|\nabla^k_{\omega_h} (u +(n+1) \log (1+c x))|_{\omega_h} = O_{k}\left(x^{-\frac{n}{2}+\frac{1}{4}-\frac{k}{2}} e^{-\frac{2\sqrt{\lambda_1}}{\sqrt{x}}}\right)\;\,\text{as}\;\,x \to 0^+.
\end{align}
Here $\lambda_1$ denotes the first eigenvalue of $D$ with respect to the flat K\"ahler metric given by the curvature form of $h^{-1}$, and neither the power $-\frac{n}{2}+\frac{1}{4}-\frac{k}{2}$ nor the exponent $2\sqrt{\lambda_1}$ can be improved.
\end{mtheorem}

In particular, $\omega$ is asymptotic to the shifted Calabi model metric $\omega_{\lambda h}$ ($\lambda = e^{-c}$) at a rate which is exponentially faster than the rate at which $\omega_{\lambda h}$ is asymptotic to $\omega_h$, i.e., exponential vs.~polynomial in terms of $x$, or doubly vs.~singly exponential in terms of the distance function $\rho$.

\subsection*{Overview of the proof}

The function $x$ is the unique decaying radial solution to the homogeneous linearized K\"ahler-Einstein equation. We expand $u$ according to powers of $x$, which results in a formal power series solution differing from $u$ by $O(x^\infty)$. This is similar to asymptotic expansion arguments as in \cite{JiangShi,LM,RZ,Schu,Wu}, with our proof being closest in spirit to \cite{JiangShi}. The coefficients of our expansion are actually constants, the first coefficient determines all the others, and the formal series converges to $-(n+1)\log(1+cx)$ for some $c \in \R$. We think of this expression as the ``tangent cone'' of $u$.

Asymptotic expansion methods \cite{JiangShi,LM,RZ,Schu,Wu} stop at the statement that the difference between $u$ and some formal (in this case, convergent) power series in $x$ is $O(x^\infty)$. A new point here is that we go further and prove an explicit estimate, \eqref{eq-mainthm}, for this difference, which is actually completely sharp as evidenced by the next generation of solutions to the homogeneous linearized PDE. This estimate is exponential in $x$. In situations with conical or cylindrical geometry, i.e., with an asymptotic scaling or translation symmetry, this kind of statement is reminiscent of the unique tangent cone theorem of \cite{AA}; see \cite[p.270]{LS} for a simplification of \cite{AA} using rescaling arguments, and see \cite{CT} for an application of \cite{LS} in the context of Einstein metrics. There are no useful scaling symmetries in our setting because our geometry is cuspidal (with two different collapsing scales) rather than conical or cylindrical.

We therefore develop a more effective version of the unique tangent cone argument, which is possible in our situation because we already have an explicit decay rate, $O(x^\infty)$, and our goal is to boost this to the decay rate of the next generation of solutions, $O(e^{-\delta/\sqrt{x}})$. The two main ingredients are (A) an explicit integral operator in terms of Bessel functions, which serves as a right inverse to the linearized operator, and which we need to bound with respect to weighted norms that are piecewise polynomial and piecewise exponential in $x$ (this is the most novel ingredient), and (B) uniform Schauder estimates on metric balls of radius $\sim 1$. These Schauder estimates are nontrivial because the geometry scale of $\omega_h$, where Schauder estimates follow from the standard theory on balls in $\C^n$, is only $\sim x$. Since $D$ is a torus, we can work around this problem in a fairly standard manner by using local universal covers, or quasi-coordinates \cite{CY2}, although in principle there is now a way of proving uniform Schauder estimates at large scales even if the local fundamental groups are trivial at all scales \cite{HeinTosatti}.

\subsection*{Organization of the paper}

In Section \ref{sec:bg} we carefully review the geometry of the model metrics $\omega_h$, compute the linearized operator in convenient coordinates, and state the necessary results from \cite{DFS}. In Section \ref{sec:expansion-powers}, after some preliminary estimates on the solution $u$, we expand $u$ according to powers of $x$, which determines the tangent cone. In Section \ref{sec:exp-decay}, we construct a suitable bounded right inverse to the linearized operator as an explicit integral operator in terms of Bessel functions and then carry out the ``unique tangent cone'' argument, which leads to the desired sharp estimate of the remainder.

\subsection*{Acknowledgments} HJH was partially supported by the DFG  under Germany's Excellence Strategy EXC 2044-390685587 ``Mathematics M\"unster:~Dynamics-Geometry-Structure" as well as by the CRC 1442 ``Geometry:~Deformations and Rigidity'' of the DFG.

We would also like to point out the related, nearly simultaneous but independent papers \cite{BiqGue, SZ}. In \cite{BiqGue}, the analysis of complex hyperbolic cusps is studied from a different point of view with the aim of constructing complex hyperbolic Dehn filling type metrics. This is related to our Example \ref{ex:conical}. In \cite[Section 6.5]{SZ}, unique tangent cone arguments in the spirit of Simon's method \cite{LS} are carried out in a different geometry with two-scale collapsing, resulting in estimates similar to our Main Theorem.

\section{Preliminaries}\label{sec:bg}

All the work in Sections \ref{ss:modelmetric}--\ref{ss:modelop} (computational properties of the model metric and of the linearized operator) goes through in a more general setting where $D$ is a compact $(n-1)$-dimensional Ricci-flat K\"ahler manifold rather than a flat complex torus. In this setting, the model metric $\omega_h$ is still K\"ahler-Einstein, but unless $D$ is a finite quotient of a complex torus it is not complex hyperbolic, and indeed its holomorphic sectional curvature is unbounded above and below. Unboundedness of the curvature would start causing problems in Section \ref{ss:dfs} (results of Datar-Fu-Song \cite{DFS}) but also at various points later in the paper. Except for generalizing the Monge-Amp\`ere $C^2$ estimate of \cite{DFS,Koba}, these problems should be solvable along the lines of \cite{HeinTosatti}, but the $C^2$ estimate remains a fundamental obstacle. Thus, overall, in this paper we chose to restrict ourselves to the case where $D$ is a torus.

\subsection{The model metric}\label{ss:modelmetric}

Let $\pi: L \to D$ be a negative holomorphic line bundle over a complex torus of complex dimension $n - 1$ for $n \geq 2$. Let $h$ be a Hermitian metric on $L$ such that the curvature form of the dual Hermitian metric is a flat K\"ahler metric on $D$. Then $h$ is unique up to scaling. We view $h$ as a smooth function on the total space of $L$ which is fiberwise homogeneous of degree $2$. In other words, we identify $h$ with the function $(\pi^*h)(s,s)$, where $s$ denotes the tautological section of the line bundle $\pi^*L \to L$. It is also helpful to introduce the three functions
\begin{align}\label{eq:fctns}
\sigma := -{\log h},\;\,x := \frac{1}{\sigma},\;\,\rho := -\sqrt{\frac{n+1}{2}}\log x.
\end{align}
With this understood, we define the well-known K\"ahler-Einstein model metric
\begin{align}\label{eq-model-KE}
\omega_h := -(n+1)i\partial\overline{\partial} \log \sigma
\end{align}
on the region $\{0 < h < 1\}$. See \cite{dBSp, FangFu} for two recent treatments of this metric in the literature. Fixing an arbitrary point $p \in D \subset L$, we choose local holomorphic coordinates $(z_1, \ldots, z_n) = (z', z_n)$ around $p$ on $L$ such that $\pi = z'$ in this chart and $z_i(p) = 0$ for all $i=1,\ldots,n$. In particular, $D = \{z_n  = 0\}$. We can also assume without loss of generality that the range of $z'$ is the unit ball $B \subset \C^{n-1}$ centered at the origin. Then $h = e^{-\varphi(z')}|z_n|^2$, and a simple computation yields the following lemma.

\begin{lemma}\label{lem:totallybasic}
In any local holomorphic chart $(z', z_n)$ as above, 
\begin{align}\label{eq:sympform}
\omega_h = (n+1)[-ix\partial\overline{\partial}\varphi + ix^2(\partial \varphi -\partial \log z_n) \wedge (\overline{\partial}\varphi - \overline{\partial} \log \bar{z}_n)].
\end{align}
Thus, under the convention that
\begin{align}
\omega_h = i \sum g_{j\bar{k}} dz^j \wedge d\bar{z}^k,
\end{align}
the metric coefficients are given by
\begin{align}\label{eq-gij}
g_{j\bar{k}} = -i\omega_h(\partial_j, \partial_{\bar{k}}) = (n+1)\left[-x\varphi_{j\bar{k}} + x^2\left(\varphi_j-\frac{\delta_{jn}}{z_n}\right)\left(\varphi_{\bar{k}} -\frac{\delta_{kn}}{\bar{z}_n}\right)\right].
\end{align}
\end{lemma}

Using \eqref{eq-gij} and the Ricci-flatness of the K\"ahler metric $-i\partial\overline\partial\varphi$ on $D$, one checks that $\omega_h$ is indeed a K\"ahler form on the whole region $\{0 < h < 1\}$ and that ${\rm Ric}(\omega_h) = -\omega_h$. Because the metric $-i\partial\overline\partial\varphi$ is actually flat in our setting, $\omega_h$ is complex hyperbolic, and its holomorphic sectional curvature, i.e., the Riemannian sectional curvature of $J$-invariant $2$-planes, is constant equal to $-\frac{2}{n+1}$. This is the only formal property of $\omega_h$ that changes if we allow for $D$ to be a general Calabi-Yau manifold.

We now study the natural foliation of our model space by the level sets
\begin{align}
Y_\epsilon := \{x = \epsilon^2\} = \{\varphi(z') - \log |z_n|^2 = \epsilon^{-2}\}.
\end{align}
We introduce a smooth real-valued coordinate chart $(x_\alpha,y_\alpha,\theta)$ on $Y_\epsilon$ by writing $z' = (z_1,\ldots,z_{n-1})$, $x_\alpha = {\rm Re}(z_\alpha)$ and $y_\alpha = {\rm Im}(z_\alpha)$ for all $\alpha \in \{1,\ldots,n-1\}$, and $z_n = re^{i\theta}$ as usual.

\begin{lemma}\label{l:formula}
Let $g_h$ denote the Riemannian metric associated with $\omega_h$ and define
\begin{equation}g_\epsilon := \frac{1}{n+1}\epsilon^{-2}g_h|_{Y_\epsilon}.\end{equation}
Then, with respect to the chart $(x_\alpha,y_\alpha,\theta)$, the metric $g_\epsilon$ is represented by
\begin{align}
\begin{split}
g_\epsilon &= \begin{pmatrix} (g_\epsilon)_{x_\alpha x_\beta} & (g_\epsilon)_{x_\alpha y_\beta} & (g_\epsilon)_{x_\alpha\theta}\\
(g_\epsilon)_{y_\alpha x_\beta} & (g_\epsilon)_{y_\alpha y_\beta} & (g_\epsilon)_{y_\alpha\theta}\\
(g_\epsilon)_{\theta x_\beta} & (g_\epsilon)_{\theta y_\beta} & (g_\epsilon)_{\theta\theta} \end{pmatrix} \\
&= \begin{pmatrix}
-\frac{1}{2}(\varphi_{x_\alpha x_\beta}+\varphi_{y_\alpha y_\beta}) + \frac{1}{2} \epsilon^2 \varphi_{y_\alpha}\varphi_{y_{\beta}} & -\frac{1}{2}(\varphi_{x_\alpha y_\beta} -\varphi_{y_\alpha x_\beta} ) -\frac{1}{2}\epsilon^2 \varphi_{y_\alpha}\varphi_{x_\beta}   &  \epsilon^2 \varphi_{y_\alpha}  \\  -\frac{1}{2}(\varphi_{y_\alpha x_\beta} -\varphi_{x_\alpha y_\beta} )  - \frac{1}{2}\epsilon^2 \varphi_{x_\alpha} \varphi_{y_\beta}& -\frac{1}{2}(\varphi_{y_\alpha y_\beta}+\varphi_{x_\alpha x_\beta}) +\frac{1}{2}\epsilon^2 \varphi_{x_\alpha}\varphi_{x_\beta}& -\epsilon^2\varphi_{x_\alpha}\\ 
 \epsilon^2\varphi_{y_\beta} & -\epsilon^2\varphi_{x_\beta}  & 2\epsilon^2
\end{pmatrix}.\label{eq-g-epsilon-1}
\end{split}
\end{align}
\end{lemma}

\begin{proof}
Let $\tilde{g}_h := \frac{1}{n+1}g_h$ and $\tilde{\omega}_h := \frac{1}{n+1}\omega_h$. Let $\Pi: B \times S^1 \to Y_\epsilon$ denote the fiberwise radial projection, which is a diffeomorphism. Then we need to compute the matrix representing the Riemannian metric $\Pi^*(\epsilon^{-2}\tilde{g}_h|_{Y_\epsilon})$ with respect to the product coordinate frame $\partial_{x_\alpha}, \partial_{y_\alpha}, \partial_\theta$ on $B \times S^1$.\medskip\

\noindent \emph{The $\theta\theta$ component.} Here we use that $\Pi$ is $S^1$-equivariant, which implies that
\begin{align}\label{eq:thetatheta}
\Pi_*(\partial_\theta) = \partial_\theta.
\end{align}
Thus, using also \eqref{eq:sympform} and the formula $\partial \log z_n = \frac{dr}{r} + id\theta$, 
\begin{align}
\begin{split}
\Pi^*(\epsilon^{-2}\tilde{g}_h|_{Y_\epsilon})(\partial_\theta,\partial_\theta) &= \epsilon^{-2}\tilde\omega_h(\partial_\theta, J\partial_\theta) \\
&= 2\epsilon^2 {\rm Im}[(\partial \log z_n)(\partial_\theta) \cdot (\overline{\partial}\log\bar{z}_n)(r\partial_r)]\\
&= 2\epsilon^2.\label{eq-gthetatheta}
\end{split}
\end{align}

\noindent \emph{The $\theta x_\alpha$ and $\theta y_\alpha$ components.} Again because $\Pi$ is $S^1$-equivariant, we must have that 
\begin{align}\label{e:project}
\Pi_*(\partial_{x_\alpha}) = \partial_{x_\alpha} + F_\alpha r\partial_r, \;\, \Pi_*(\partial_{y_\alpha}) = \partial_{y_\alpha} + G_\alpha r\partial_r
\end{align}
for some real-valued functions $F_\alpha,G_\alpha$. Therefore,
\begin{align}\label{eq:thetay}
\begin{split}
\Pi^*(\epsilon^{-2}\tilde{g}_h|_{Y_\epsilon})(\partial_\theta,\partial_{x_\alpha}) &=\epsilon^{-2}\tilde{\omega}_h(\partial_\theta,J(\partial_{x_\alpha} + F r\partial_r))\\
&= \epsilon^{-2}\tilde\omega_h(\partial_\theta,\partial_{y_\alpha})\\
&= -2\epsilon^2 {\rm Im}[(\partial\varphi)(\partial_{y_\alpha}) \cdot (\overline{\partial}\log \bar{z}_n)(\partial_\theta)]\\
&= \epsilon^2\varphi_{y_\alpha}.
\end{split}
\end{align}
Here we have used the general formula $\partial\varphi = \frac{1}{2}(d\varphi - i d\varphi \circ J)$. Similarly,
\begin{align}\label{eq:thetax}
\Pi^*(\epsilon^{-2}\tilde{g}_h|_{Y_\epsilon})(\partial_\theta,\partial_{y_\alpha}) &= -\epsilon^2 \varphi_{x_\alpha}.
\end{align}

\noindent \emph{The $x_\alpha x_\beta$, $x_\alpha y_\beta$ and $y_\alpha y_\beta$ components.} Here we first need to determine the functions $F_\alpha,G_\alpha$ of \eqref{e:project}. Notice that $\Pi_*(\partial_{x_\alpha}), \Pi_*(\partial_{y_\alpha})$ must be tangent to $Y_\epsilon$ and the tangent bundle of $Y_\epsilon$ is the kernel of the differential of the function $\varphi(z') - \log |z_n|^2$. Using this, it is easy to see that
\begin{align}\label{e:project2}
F_\alpha = \frac{1}{2}\varphi_{x_\alpha}, \;\, G_\alpha = \frac{1}{2}\varphi_{y_\alpha}.
\end{align}
Thus,
\begin{align}
\begin{split}
\Pi^*(\epsilon^{-2}&\tilde{g}_h|_{Y_\epsilon})(\partial_{x_\alpha},\partial_{x_\beta}) = \epsilon^{-2}\tilde\omega_h(\partial_{x_\alpha} + F_\alpha r\partial_r, \partial_{y_\beta} + F_\beta \partial_\theta)\\
&= (-i\partial\overline{\partial}\varphi)(\partial_{x_\alpha},\partial_{y_\beta})-2\epsilon^2{\rm Im}\left[\left((\partial\varphi)(\partial_{x_\alpha}) - \frac{1}{2}\varphi_{x_\alpha}\right) \cdot \left((\overline{\partial}\varphi)(\partial_{y_\beta}) + \frac{i}{2}\varphi_{x_\beta}\right)\right]\\
&=-\frac{1}{2}(\varphi_{x_\alpha x_\beta} + \varphi_{y_\alpha y_\beta}) + \frac{1}{2}\epsilon^2\varphi_{y_\alpha}\varphi_{y_\beta}.
\end{split}
\end{align}
Here we have applied the general K\"ahler geometry formula
\begin{align}
(i\partial\overline{\partial}\varphi)(U,JV) = \frac{1}{2}((\nabla^2 \varphi)(U,V) + (\nabla^2 \varphi)(JU,JV))
\end{align}
with respect to the standard Euclidean K\"ahler metric on $B$. Similarly,
\begin{align}
\begin{split}
\Pi^*(\epsilon^{-2}&\tilde{g}_h|_{Y_\epsilon})(\partial_{x_\alpha},\partial_{y_\beta}) = \epsilon^{-2}\tilde\omega_h(\partial_{x_\alpha} + F_\alpha r\partial_r, -\partial_{x_\beta} + G_\beta \partial_\theta)\\
&= (-i\partial\overline{\partial}\varphi)(\partial_{x_\alpha},-\partial_{x_\beta})-2\epsilon^2{\rm Im}\left[\left((\partial\varphi)(\partial_{x_\alpha}) - \frac{1}{2}\varphi_{x_\alpha}\right) \cdot \left((\overline{\partial}\varphi)(-\partial_{x_\beta}) + \frac{i}{2}\varphi_{y_\beta}\right)\right]\\
&=-\frac{1}{2}(\varphi_{x_\alpha y_\beta} - \varphi_{y_\alpha x_\beta}) - \frac{1}{2}\epsilon^2\varphi_{y_\alpha}\varphi_{x_\beta},
\end{split}
\end{align}
as well as
\begin{align}
\begin{split}
\Pi^*(\epsilon^{-2}&\tilde{g}_h|_{Y_\epsilon})(\partial_{y_\alpha},\partial_{y_\beta}) = \epsilon^{-2}\tilde\omega_h(\partial_{y_\alpha} + G_\alpha r\partial_r, -\partial_{x_\beta} + G_\beta \partial_\theta)\\
&= (-i\partial\overline{\partial}\varphi)(\partial_{y_\alpha},-\partial_{x_\beta})-2\epsilon^2{\rm Im}\left[\left((\partial\varphi)(\partial_{y_\alpha}) - \frac{1}{2}\varphi_{y_\alpha}\right) \cdot \left((\overline{\partial}\varphi)(-\partial_{x_\beta}) + \frac{i}{2}\varphi_{y_\beta}\right)\right]\\
&=-\frac{1}{2}(\varphi_{y_\alpha y_\beta} + \varphi_{x_\alpha x_\beta}) + \frac{1}{2}\epsilon^2\varphi_{x_\alpha}\varphi_{x_\beta}.
\end{split}
\end{align}
This concludes the proof of the lemma.
\end{proof}

So far we have described the restriction of the metric $g_h$ from the total space $L$ to the level sets $Y_\epsilon$. We still need to calculate the $g_h$-unit normal of $Y_\epsilon$ in the given coordinates on $L$. This turns out to be  straightforward. The answer also tells us that the $g_h$-gradient of the function $\rho = -\sqrt{(n+1)/2}\log x$ has constant length $1$, so in particular the leaves of our foliation are equidistant.

\begin{lemma}\label{lem-n}
The $g_h$-unit normal of $Y_\epsilon$ pointing towards the cuspidal end is
\begin{align}\label{eq:lem-n}
\mathbf{n} = -\frac{\epsilon^{-2}}{\sqrt{2(n+1)}}r \partial_r.
\end{align}
Here $r\partial_r  =- J\partial_\theta $ is globally defined because $\partial_\theta$ generates the canonical $S^1$-action on $L$.
\end{lemma}

\begin{proof}
We first prove that $r\partial_r$ is $g_h$-orthogonal to the tangent space of $Y_\epsilon$, which is locally spanned by the vector fields $\Pi_*(\partial_\theta), \Pi_*(\partial_{x_\alpha}), \Pi_*(\partial_{y_\alpha})$ from the proof of Lemma \ref{l:formula}. Indeed, by  \eqref{eq:thetatheta},
\begin{align}
g_h(r\partial_r, \Pi_*(\partial_\theta)) = -\omega_h(\partial_\theta, \partial_\theta) = 0.
\end{align}
Secondly, by \eqref{e:project}, \eqref{e:project2} and \eqref{eq:thetax}, \eqref{eq-gthetatheta},
\begin{align}
g_h(r\partial_r, \Pi_*(\partial_{x_\alpha})) &= \omega_h\left(\partial_{x_\alpha}  + \frac{1}{2}\varphi_{x_\alpha} r\partial_r,\partial_\theta\right) = (n+1)\left(-\epsilon^4 \varphi_{x_\alpha} + \frac{1}{2}\varphi_{x_\alpha} \cdot 2\epsilon^4\right) = 0,
\end{align}
and orthogonality to $\Pi_*(\partial_{y_\alpha})$ is similar.
Thus, $\mathbf{n} = f r\partial_r$ for some function $f$, where
\begin{align}
1 = \omega_h(\mathbf{n} , J \mathbf{n}) = f^2 (n+1) \tilde{\omega}_h(\partial_\theta,J\partial_\theta) = 2(n+1)\epsilon^4 f^2
\end{align}
again thanks to \eqref{eq-gthetatheta}.
\end{proof}

The following lemma is not necessary for this paper but we find it helpful for intuition. It says that the orbits of the unit normal field $\mathbf{n}$ of our foliation are geodesics and the leaves have constant mean curvature independent of $\epsilon$. This also holds if $D$ is a non-flat Calabi-Yau manifold. If $D$ is flat, these are well-known properties of the universal covers of the leaves, which are horospheres in $\C H^n$.

\begin{lemma}\label{l:busemann}
We have $\nabla_{\mathbf{n}} \mathbf{n} = 0$ and the mean curvature of $Y_\epsilon$ with respect to $\mathbf{n}$ is $-n/\hspace{-0.5mm}\sqrt{2(n+1)}$.
\end{lemma}

\begin{proof}
Replacing $\epsilon^{-2} = \sigma$ in \eqref{eq:lem-n}, differentiating, and using that $(r\partial_r)(\sigma) = -2$, we get
\begin{align}
2(n+1)\nabla_{\mathbf{n}} \mathbf{n} = -2\sigma r\partial_r + \sigma^2 \nabla_{r\partial_r}(r \partial_r).
\end{align}
On the other hand, $\langle \nabla_{\mathbf{n}}\mathbf{n},\mathbf{n}\rangle = 0$. Thus, it suffices to prove that $\nabla_{r\partial_r}(r \partial_r)$ is orthogonal to the tangent space of $Y_\epsilon$. For this, note that $r\partial_r$ is a holomorphic vector field, so its covariant derivative is $J$-linear. Because $\nabla J = 0$, we conclude that $\nabla_{r\partial_r}(r \partial_r) = - \nabla_{\partial_\theta}\partial_\theta$, so it suffices to prove that $\langle \nabla_{\partial_\theta}\partial_\theta, \mathbf{t}\rangle = 0$ for all vectors $\mathbf{t}$ tangent to $Y_\epsilon$. This can be done as follows. Since $\partial_\theta$ is a Killing vector field,
\begin{align}
\langle \nabla_{ {\partial_\theta}}\partial_\theta, \mathbf{t} \rangle = -\langle\nabla_\mathbf{t} \partial_\theta, \partial_\theta \rangle
= -\frac{1}{2} \mathbf{t} \langle\partial_\theta, \partial_\theta \rangle =-(n+1) \mathbf{t} (\sigma^2) = 0,
\end{align}
as desired, using also \eqref{eq-gthetatheta}. This proves that $\nabla_{\mathbf{n}}\mathbf{n} = 0$.

Next, we compute the mean curvature of $Y_\epsilon$. The unoriented area element of $Y_\epsilon$ is 
\begin{align}\label{eq:dA}
dA = \sqrt{\det (g_h|_{Y_\epsilon})}\, d\theta \, dx_\alpha \, dy_\alpha = (n+1)^{\frac{n-1}{2}}\epsilon^{n-1}\sqrt{\det (g_\epsilon)}\,d\theta \, dx_\alpha \, dy_\alpha.
\end{align}
By \eqref{eq-g-epsilon-1} and column operations, $\det (g_\epsilon)$ equals the determinant of 
\begin{align}
 \begin{pmatrix}
-\frac{1}{2}(\varphi_{x_\alpha x_\beta}+\varphi_{y_\alpha y_\beta})& -\frac{1}{2}(\varphi_{x_\alpha y_\beta} -\varphi_{y_\alpha x_\beta} )    &  \epsilon^2 \varphi_{y_\alpha}  \\  -\frac{1}{2}(\varphi_{y_\alpha x_\beta} -\varphi_{x_\alpha y_\beta} )  & -\frac{1}{2}(\varphi_{y_\alpha y_\beta}+\varphi_{x_\alpha x_\beta})& -\epsilon^2\varphi_{x_\alpha}\\ 
0 & 0 & 2\epsilon^2
\end{pmatrix},
\end{align}
which equals $\epsilon^2 f(x_\alpha, y_\alpha)$ with $f$  smooth and independent of $\epsilon$. Thus,
\begin{align}
\frac{\mathcal{L}_{\mathbf{n}} dA}{dA} = \frac{\mathbf{n}(\sigma^{-\frac{n}{2}})}{\sigma^{-\frac{n}{2}}} = \frac{\frac{n}{2} \cdot (r\partial_r)(\sigma)}{\sqrt{2(n+1)}} = -\frac{n}{\sqrt{2(n+1)}},
\end{align}
and this is the mean curvature of $Y_\epsilon$ with respect to $\mathbf{n}$.
\end{proof}

\subsection{The model operator}\label{ss:modelop}

We now compute the linearization of the Monge-Amp\`ere operator at our model metric $\omega_h$. With $(g_{j\bar{k}})$ as in \eqref{eq-gij}, this is the operator
\begin{align}\label{eq:defineLh}
L_h(u) := \frac{d}{dt}\biggr|_{t=0} \frac{(\omega_h + \ddbar(tu))^n}{e^{tu}\omega_h^n} = g^{\bar{k}j}  u_{j \bar{k}} - u.
\end{align}
Fix any local holomorphic chart $(z',z_n)$ on the total space as in Section \ref{ss:modelmetric}, so that $h = e^{-\varphi(z')}|z_n|^2$. Consider the associated smooth chart $(x_\alpha,y_\alpha,x,\theta)$, where $x = \frac{1}{\sigma}$. In this chart, we can still define the complex partials $u_\alpha = \partial_{z_\alpha}u$ and $u_{\bar\alpha}= \partial_{\bar{z}_\alpha}u$ in the usual way for all $\alpha\in\{1,\ldots,n-1\}$.

\begin{lemma}\label{lem:Lu}
In any chart $(x_\alpha,y_\alpha,x,\theta)$ as above, the linearized operator is given by
\begin{align}\label{eq-Lu}
\begin{split}
L_h(u)   = \frac{1}{n+1}(x^2 u_{xx} + (n+1)xu_x -(n+1) u  -x^{-1} \varphi^{\bar{\beta} \alpha} u_{\alpha \bar \beta}  + (2x)^{-2} u_{\theta \theta})  + F,\\
F = F(x_\alpha,y_\alpha, x^{-1} u_{\theta\alpha}, x^{-1} u_{\theta\bar\alpha},x^{-1} u_{\theta\theta}),
\end{split}
\end{align}
where $F$ is smooth in $x_\alpha,y_\alpha$ and linear homogeneous in the other arguments.
\end{lemma}

\begin{proof}
In the chart $(z',z_n)$, the definition \eqref{eq:defineLh} of $L_h$ can be expanded into
\begin{align}\label{eq-veryhelpful}
L_h (u) = g^{\bar{k}j}  u_{j \bar{k}} - u = g^{\bar \beta \alpha} u_{\alpha \bar \beta} +2{\rm Re}( g^{\bar{\alpha}n}  u_{\bar \alpha n}) + g^{\bar n n} u_{n \bar n}   - u.
\end{align}
With $Q := \varphi^{\bar{\beta} \gamma}  \varphi_{\bar{\beta}}  \varphi_\gamma$, the inverse of the metric tensor is given for $\alpha,\beta\in\{1,\ldots,n-1\}$ by
\begin{align}\label{eq-ginverse}
g^{\bar{\beta} \alpha} = -\frac{\varphi^{\bar{\beta} \alpha}}{(n+1)x},\;\,  g^{\bar \alpha n} =  -\frac{\varphi^{\bar{\alpha} \gamma}\varphi_\gamma z_n}{(n+1)x}, \;\, g^{\bar n n} = \frac{r^2(1 - Q x)}{(n+1)x^2}.
\end{align}
This can be checked by direct multiplication, using \eqref{eq-gij}.

Consider the change of coordinates $(z', z_n) \mapsto (x_\alpha, y_\alpha, x, \theta)$. A simple calculation shows that
\begin{align}\label{eq-partial}
\partial_{z_\alpha} \mapsto \partial_{z_\alpha}  -x^2 \varphi_\alpha \partial_x, \;\, \partial_{z_n} \mapsto \frac{x^2}{z_n}\partial_x - \frac{i}{2z_n}\partial_\theta.
\end{align}
Thus, under the new coordinates $(x_\alpha,y_\alpha, x, \theta)$, 
\begin{align}\label{eq-careful1}
\begin{split}
(n+1)g^{\bar \beta \alpha} u_{\alpha \bar \beta} &=  -x^{-1}\varphi^{\bar{\beta} \alpha}( \partial_{z_\alpha}   -x^2 \varphi_\alpha \partial_x )( u_{\bar\beta}  - x^2 \varphi_{\bar \beta} u_x)\\
& =  -x^{-1} \varphi^{\bar{\beta} \alpha} u_{\alpha \bar \beta} + (n-1) xu_x + \boxed{x\varphi^{\bar\beta\alpha}(\varphi_\alpha u_{x\bar\beta} + \varphi_{\bar\beta} u_{x\alpha})} - \underline{\underline{Q(x^3  u_{xx} + 2x^2 u_x)}}.
\end{split}
\end{align}
Similarly, with ${F}_1$ as in the statement of the lemma,
\begin{align}\label{eq-careful2}
\begin{split}
(n+1)g^{\bar{\alpha} n }  u_{ n \bar \alpha} &= - \frac{z_n}{x} \varphi^{\bar\alpha\gamma} \varphi_\gamma \left(\frac{x^2}{z_n} \partial_x - \frac{i}{2z_n} \partial_\theta \right)  \left(u_{\bar\alpha}  -x^2 \varphi_{\bar\alpha} u_x \right)\\
&= \boxed{-x\varphi^{\bar\alpha\gamma}\varphi_{\gamma} u_{x\bar\alpha}} + \underline{\underline{Q(x^3 u_{xx} +2x^2u_x)}} - \frac{i}{2} Q x u_{x\theta} + {F}_1(x_\alpha,y_\alpha, x^{-1}u_{\theta\bar\alpha}).
\end{split}
\end{align} The last contribution is slightly more complicated:
\begin{align}\label{eq-careful3}
\begin{split}
(n+1)g^{\bar{n}n}  u_{n \bar{n}} &=\frac{r^2}{x^2}(1 - Qx) \left(\frac{x^2}{z_n} \partial_x - \frac{i}{2z_n} \partial_\theta  \right) \left(\frac{x^2}{\bar{z}_n}u_x + \frac{i}{2\bar{z}_n} u_\theta \right)\\
&= x^2 u_{xx} + 2x u_x +(2x)^{-2} u_{\theta \theta} - \underline{\underline{Q(x^3  u_{xx} + 2x^2 u_x)}} + {F}_2(x_\alpha,y_\alpha, x^{-1}u_{\theta\theta}).
\end{split}
\end{align}
The delicate part here is the coefficients of $u_x$ and $u_\theta$, which arise as follows:
\begin{align}
\frac{r^2}{x^2}\left(\frac{x^2}{z_n}\partial_x -\frac{i}{2z_n}\partial_\theta\right)\left(\frac{x^2}{\bar{z}_n}\right) &= \frac{r^2}{x^2}\partial_{z_n}\left(\frac{x^2}{\bar{z}_n}\right) = \frac{r^2}{x^2}\frac{1}{\bar{z}_n}\partial_{z_n}(\varphi - \log |z_n|^2)^{-2} = 2x
\end{align}
for $u_x$, and similarly the coefficient of $u_\theta$ vanishes.

When we add the contributions of \eqref{eq-careful1}, \eqref{eq-careful2} and \eqref{eq-careful3} according to \eqref{eq-veryhelpful}, the boxed and the underlined terms cancel. Cancellation of the boxed terms will be important in Section \ref{sec:exp-decay}.
\end{proof}

\subsection{Results of Datar-Fu-Song}\label{ss:dfs}

In the setting of the Main Theorem, consider a member $h$ of the canonical $1$-parameter family of Hermitian metrics on $L \to D$ and its associated model metric $\omega_h$. For any tubular neighborhood $V$ of the zero section $D \subset L$ of the form $V = \{h \leq \delta\}$ for some $\delta \in (0,1)$, a special case of Datar-Fu-Song \cite[Thm 4.1]{DFS} says that for all $\beta\in C^\infty(\partial V)$, the Dirichlet problem 
\begin{align}
\begin{split}\label{eq:dirprob}
(\omega_h + i\partial\overline{\partial} u)^n = e^u \omega_h^n\;\,\text{on}\;\,V\setminus D,\\
u = \beta\;\,\text{on}\;\,\partial V,
\end{split}
\end{align}
has a smooth solution $u$ such that the K\"ahler-Einstein metric $\omega_h+ i\partial\overline{\partial} u$ on $V \setminus D$ is complete.

Conversely assume that for some closed tubular neighborhood $U$ of the zero section, there exists a K\"ahler-Einstein metric $\omega$ on $U \setminus D$ such that ${\rm Ric}(\omega) = -\omega$. Thanks to the K\"ahler-Einstein equation, we have that  $\omega = \omega_h + i\partial\overline{\partial} u$, where $u$ is the log volume ratio of the two metrics. Equivalently,
\begin{align}\label{1}
 (\omega_h + i\partial \overline{\partial} u)^n = e^{u} \omega_h^n.
\end{align}
Then the following result is a special case of \cite[Thm 1.3]{DFS}. This still holds in the unbounded curvature setting where $D$ is a non-flat Calabi-Yau manifold.

\begin{theorem}\label{thm:DFS0}
If $\omega$ is complete, then there exists a constant $C$ such that for all $0 < \eta \leq 1$ we have that $|u(p)| \leq C\eta$ for all $p \in U\setminus D$ such that ${\rm dist}_{\omega_h}(p,\partial U) \geq \eta^{-1}$ and ${\rm dist}_{\omega}(p,\partial U) \geq \eta^{-1}$.
\end{theorem}

In particular, the maximum principle can now be applied to prove that $u$ is the \emph{unique} solution to the Dirichlet problem \eqref{eq:dirprob} such that $\omega_h + \ddbar u$ is complete, where $\beta$ is defined to be the restriction of $u$ to $\partial V$ for any $V$ as above such that $V \subset U$. This completeness assumption cannot be dropped, and \eqref{1} does not guarantee by itself that $u$ is bounded, as the following example shows.

\begin{example}\label{ex:conical}
We use the Calabi ansatz to construct K\"ahler-Einstein metrics on $V \setminus D = \{0 < h \leq \delta\}$ for some $\delta>0$. Let $\omega=\ddbar \psi(-\sigma)$, where $\psi: (-\infty, \log \delta] \to \R$ is smooth. Then $\omega$ is positive definite if and only if $\psi' > 0$ and $\psi'' > 0$, and the equation ${\rm Ric}(\omega) = -\omega$ is equivalent to
\begin{align}\label{eq-calabiODE}
(\psi')^{n-1}\psi'' = e^{\psi + a}
\end{align}
for some constant $a \in \R$. Our standard solutions $\omega = \omega_{\lambda h}$ ($\lambda \in \R^+$) in this paper correspond to 
\begin{align}\label{eq-standardsolns}
\psi(t) = -(n+1)\log(-(t+\log\lambda)), \;\, e^a = (n+1)^n,
\end{align}
for $\delta \in (0, \frac{1}{\lambda})$. However, there do exist other solutions which are defined on all of $V \setminus D$.

In fact, by multiplying \eqref{eq-calabiODE} by $\psi' > 0$ and integrating, we see that \eqref{eq-calabiODE} is equivalent to
\begin{align}\label{eq-modified}
\frac{1}{n+1}(\psi')^{n+1} = e^{\psi + a} + b
\end{align}
for some constant $b \in \R$. This is in turn equivalent to
\begin{align}\label{eq-integrated}
\int_{\psi(t)}^{\psi(t_0)} (e^{\xi + a} + b)^{-\frac{1}{n+1}}d\xi = (n+1)^{\frac{1}{n+1}}(t_0 - t)
\end{align}
for all $t \leq t_0$ such that $\psi$ exists on $[t,t_0]$ and $e^{\psi + a} + b$ is strictly positive. If $b = 0$, then \eqref{eq-integrated} implies \eqref{eq-standardsolns} up to an additive constant. For $b < 0$, the integral in \eqref{eq-integrated} is uniformly bounded independently of $\psi(t)$, so the interval of allowed $t$-values has a finite lower bound. For $b > 0$, this obstruction does not arise, so $\psi$ exists on $(-\infty,\log\delta]$ for some $\delta > 0$ and $\psi'$ $>$ $c := ((n+1)b)^{1/(n+1)} > 0$ and $\psi'' > 0$.

The property $\psi' > c$ together with the mean value theorem tells us that $\psi(t) \leq \psi(t_0) + c(t - t_0)$ for all $t \leq t_0$. Then \eqref{eq-calabiODE} implies $\psi''(t) = O(e^{ct})$ and \eqref{eq-modified} implies $\psi'(t) - c = O(e^{c t})$ as $t \to -\infty$. Using this, one can check that $(V\setminus D, \omega)$ has finite diameter. In fact, $\omega$ defines a singular metric on $V$ with conical singularities along $D$, where $D$ has diameter proportional to $\sqrt{c}$ and the cone angle is $2\pi c$. To see that the angle is exactly $2\pi c$, notice that $\psi'(t) = c + O(e^{ct})$ implies $\psi(t) = ct + d + O(e^{ct})$ for some constant $d \in \R$, and then $\psi''(t)$ is asymptotic to a constant multiple of $e^{ct}$ by \eqref{eq-calabiODE}.
\end{example}

\begin{corollary}\label{cor:clarify-DFS}
If $\omega$ is complete and if in addition $\sup_{U \setminus D} |\omega|_{\omega_h} < \infty$, then $\omega$ is uniformly equivalent to $\omega_h$ and we have that $|u| = O((-{\log x})^{-1})$ as $x \to 0^+$. 
\end{corollary}

\begin{proof}
By Theorem \ref{thm:DFS0}, $u$ is bounded, i.e., the volume forms $\omega_h^n$ and $\omega^n$ are uniformly comparable to each other. Since $|\omega|_{\omega_h}$ is bounded, it follows that $\omega$ and $\omega_h$ are themselves uniformly equivalent. Then $|u(p)| = O((-{\log x(p)})^{-1})$ follows for all $p \in U \setminus D$ by taking $\eta$ in Theorem \ref{thm:DFS0} to be comparable to $(-{\log x(p)})^{-1}$ because the gradient of $-{\log x(\cdot)}$ has constant $\omega_h$-length by Lemma \ref{l:busemann}.
\end{proof}

If $D$ is a finite quotient of a complex torus,  then the bound $\sup_{U \setminus D} |\omega|_{\omega_h} < \infty$ can actually be proved \cite{DFS,Koba}, which is also the reason why the Dirichlet problem \eqref{eq:dirprob} can be proved to admit a complete solution in this case \cite{DFS}. Applying the standard local regularity theory of the complex Monge-Amp\`ere equation in quasi-coordinates \cite{CY}, we thus obtain the following slight refinement of \cite[Thm 1.4]{DFS}.

\begin{theorem}\label{thm-DFS}
If $D$ is a finite quotient of a torus and if $\omega$ is complete, then $\omega$ is uniformly equivalent to $\omega_h$ and we have for all $k \in \N_0$ that
$|\nabla^k_{\omega_h}u|_{\omega_h} = O_k((-{\log x})^{-1})$ as $x\to 0^+$.
\end{theorem}

\section{Expansion according to powers of $x$}\label{sec:expansion-powers}

Assume the same setup as in the Main Theorem. Since $\omega_h$ and $\omega$ are K\"ahler-Einstein with the same Einstein constant, it follows that $\omega = \omega_h + i\partial\overline{\partial} u$, where $u$ is the log volume ratio. Since $\omega$ is complete, it follows from Theorem  \ref{thm-DFS} that $|u| = O((-{\log x})^{-1})$ with all derivatives. Also, by definition,
\begin{align}\label{eq-main}
(\omega_h + i\partial \overline{\partial} u)^n = e^{u} \omega_h^n.
\end{align}
In this section we prove the following theorem, which is our first step towards the Main Theorem.

\begin{theorem}\label{thm:expansion-powers}
There exists a constant $c \in \R$ such that 
\begin{align}\label{eq:expansion-powers}
u + (n+1)\log(1 + cx) = {O}(x^\infty).
\end{align}
\end{theorem}

The proof will be given in Section \ref{ss:proof-3.1} after some preliminary estimates in Section \ref{ss:prelim-ests}.

Note that if \eqref{eq:expansion-powers} holds in the pointwise sense, then it automatically holds for all derivatives. This follows from the classical regularity theory of the Monge-Amp\`ere equation. For this it actually suffices to work at the scale of the  injectivity radius because the constants in the elliptic estimates only blow up like fixed negative powers of $x$ even at this scale. However, during the proof of Theorem \ref{thm:expansion-powers}, it is helpful to work with a more specific $O$ notation, which we will now define; cf.~\cite[Thm 1.1]{JiangShi}.

\begin{definition}\label{defn-boldO}
After shrinking $U$ if necessary, choose a finite cover of $U \setminus D$ by charts $(x_\alpha,y_\alpha, x, \theta)$ as in Sections \ref{ss:modelmetric}--\ref{ss:modelop}. Given this, we say that a function $f \in C^\infty(U \setminus D)$ is $\mathbf{O}(x^k)$ for some $k \in \R$ if, under all local charts of this cover, for all $c\in \mathbb{N}_0$ there exists a constant $C$ such that
\begin{align}\label{eq:fatO1}
|x^c\partial^c_x f|\leq Cx^{k},
\end{align}
and for all $a,b,c,d, N \in \N_{0}$ with $a + b + d \geq 1$ there exists a constant $C'$ such that
\begin{align}\label{eq:fatO2}
|\partial^a_{x^\alpha} \partial^b_{y^\alpha} \partial^c_x \partial^d_\theta f|\leq C'x^{N}.
\end{align}
\end{definition}

\subsection{Preliminary estimates}\label{ss:prelim-ests} 

We begin by proving that $|u| = O(x)$ as $x \to 0^+$. This bound is sharp because $x$ is the unique decaying radial solution to the homogeneous linearized PDE, as one can easily check using Lemma \ref{lem:Lu}. Our proof uses a barrier argument as in \cite[Lemma 2.2]{JiangShi}, and this argument would go through for an arbitrary compact Calabi-Yau manifold $D$.

\begin{lemma}\label{lem-u-c0}
Assuming only that $|u| = o(1)$ as $x \to 0^+$, we obtain that $|u|= O(x)$.
\end{lemma}

\begin{proof} 
Define $U_\delta = \{0 < x \leq \delta\} \subset L \setminus D$ with $\delta \in (0,1)$ small enough so that $U_\delta \subset U\setminus D$. We will prove the lemma in two steps, using a suitable barrier function in each step. The key is that $x^{1-\epsilon}$ is a supersolution of the linearized operator $L_h$ for all $\epsilon \in (0,1)$ (in fact, for all $\epsilon \in (0,n+2)$, but we are only interested in functions that go to zero as $x \to 0^+$) and $x^{1+\epsilon}$ is a subsolution for all $\epsilon>0$.

In the first step, we consider the barrier $B := x^{1-\epsilon}$ for an arbitrary but fixed $\epsilon \in (0,1)$. Let
\begin{align}\label{eq:bdry}
a_\delta := \max_{\partial U_\delta} \frac{|u|}{B} = o(\delta^{\epsilon-1})\;\,\text{as}\;\,\delta\to 0.
\end{align}
Then $u - a_\delta B \leq 0$ on $\partial U_\delta$. We now prove that $u - a_\delta B \leq 0$ on $U_\delta$ if $\delta$ is small. If this was false, then, because $u = o(1)$ as $x \to 0^+$, for arbitrarily small values of $\delta$ there would exist a point $p_\delta \in {\rm int}\,U_\delta$ such that $u-a_\delta B$ attains a positive global maximum at $p_\delta$. Thus, at $p_\delta$,
\begin{align}\label{eq:barr1}
0 = \log\left(\frac{(\omega_h + \ddbar u)^n}{\omega_h^n}\right) - u < \log\left(\frac{(\omega_h + \ddbar (a_\delta B))^n}{\omega_h^n}\right) - a_\delta B.
\end{align}
Now $|\ddbar B|_{\omega_h} = c B$ for some constant $c = c(n,\epsilon)$, so by \eqref{eq:bdry} we have $|\ddbar(a_\delta B)|_{\omega_h} = o(\delta^{\epsilon-1}) B = o(1)$ uniformly on $U_\delta$ as $\delta \to 0$. This allows us to linearize the right-hand side of \eqref{eq:barr1} if $\delta$ is small enough. Thus, there exists a constant $C = C(n)$ such that if $\delta$ is small, then, uniformly on $U_\delta$, 
\begin{align}\label{eq:barr2}
\begin{split}
\log\left(\frac{(\omega_h + \ddbar (a_\delta B))^n}{\omega_h^n}\right) - a_\delta B &\leq L_h(a_\delta B) + C|\ddbar(a_\delta B)|_{\omega_h}^2\\
&= -\epsilon(n+2-\epsilon)(a_\delta B) + c^2C (a_\delta B)^2.
\end{split}
\end{align}
Again because $\sup_{U_\delta} (a_\delta B) = o(1)$ as $\delta \to 0$, this is eventually nonpositive on all of $U_\delta$, contradicting \eqref{eq:barr1} at $p_\delta$. Thus, fixing $\delta$ small enough depending on $u,n,\epsilon$, we have that $u \leq a_\delta x^{1-\epsilon} $ on $U_\delta$.

A lower bound can be proved in the same way. The only difference is that the uniform smallness of $a_\delta B$ on $U_\delta$ is then also needed to ensure that $\omega_h - \ddbar(a_\delta B) > 0$. Notice also that, in the proof of the upper bound, the quadratic terms on the right-hand side of \eqref{eq:barr2} were actually unnecessary because the Monge-Amp\`ere operator is concave, but they are necessary in the proof of the lower bound.

In the second step, we carry out the same argument using $B := x - x^{1+\epsilon}$ as a barrier for any fixed $\epsilon \in (0,1)$. This is a positive supersolution to $L_h$ on $U_\delta$ for all $\epsilon>0$ because $x$ is a solution and $x^{1+\epsilon}$ is a subsolution. The only difference to the first step is that the good linear term on the right-hand side of \eqref{eq:barr2} is now comparable to $x^\epsilon (a_\delta B)$, but the nonlinearities are still only bounded by $(a_\delta B)^2$. Thus, we can complete the argument only if we know that $\sup_{U_\delta} (x^{-\epsilon}(a_\delta B)) = o(1)$ as $\delta\to 0$. But this is true as long as $\epsilon<1$. The reason is that $|u| = O(x^\mu)$ as $x \to 0^+$ for all $\mu \in (0,1)$ from the first step, which implies $a_\delta  = \max_{\partial U_\delta}(|u|/B) = O(\delta^{\mu-1})$ as $\delta \to 0$. Thus, it suffices to choose $\mu \in (\epsilon,1)$.
\end{proof}

Sharp higher-order bounds for $u$ are now almost immediate as in \cite[Lemma 2.4]{JiangShi}. Here we use the flatness of $D$ although it seems likely that this can be avoided using the methods of \cite{HeinTosatti}.

\begin{lemma}\label{lem-v_higher}
Assuming only that $|u| = O(x)$ as $x \to 0^+$ and that $\omega$ is uniformly equivalent to $\omega_h$, we obtain that $|\nabla^k_{\omega_h}u|_{\omega_h}=O_k(x)$  for all $k \in \N$.
\end{lemma}

\begin{proof}
Shrinking $U$ slightly if necessary, we can assume that $\partial U = \{h = \delta\}$ for some $\delta \in (0,1)$.

Consider a local chart $(x_\alpha,y_\alpha, x, \theta)$ on $U \setminus D$ as in Sections \ref{ss:modelmetric}--\ref{ss:modelop}. Thus, for some $p \in D \subset U$, we have local holomorphic coordinates $(z', z_n)$ on $U$ mapping $p$ to $0 \in \mathbb{C}^n$ such that the bundle projection $\pi$ is given by $z \mapsto z' = (z_1,\ldots,z_{n-1})$. In terms of these coordinates, $x_\alpha = {\rm Re}(z_\alpha)$ and $y_\alpha = {\rm Im}(z_\alpha)$ for all $\alpha \in \{1,\ldots,n-1\}$, and $z_n = r e^{i\theta}$ and $x = (-{\log h})^{-1} = (\varphi(z') - \log r^2)^{-1}$.

Because $D$ is flat, there exists a holomorphic covering map $\mathbb{C}^{n-1} \to D$. The line bundle $L$ becomes holomorphically trivial after pulling back under this map. We can then assume that $z'$ is the standard coordinate vector on $\mathbb{C}^{n-1}$ and that $L$ has been trivialized in such a way that $\varphi(z')$ is a homogeneous quadratic polynomial. Moreover, after passing to the universal cover of the set $\{z_n \neq 0\}$, we may view $\theta$ as a globally defined function. In this way, we realize the universal cover of $U \setminus D$ as a subset of $\mathbb{C}^n$ with holomorphic coordinates $(z', w)$, where $w := \log r + i\theta$. We also get a smooth chart $(x_\alpha,y_\alpha,x,\theta)$ identifying the universal cover of $U \setminus D$ with $\mathbb{R}^{2n-2} \times (0,x_\delta] \times \R$, where $x_\delta := (-{\log \delta})^{-1}$.

Fix $q \in U \setminus D$ such that $c_n x(q) < x_\delta$, where $c_n := \exp(\sqrt{2/(n+1)})$. Let $B_1(q)$ be the Riemannian unit ball with respect to $\omega_h$ centered at $q$. Then, using our distance function $\rho$ from \eqref{eq:fctns},
\begin{align}\label{eq:ballsqueeze}
B_1(q) \subset \{\rho(q)-1<\rho<\rho(q)+1\} = \{ c_n^{-1} x(q) < x < c_n x(q) \} \subset U\setminus D.
\end{align}
On the universal cover of $U \setminus D$, we define a new chart via
\begin{align}\label{eq-quasi-coor}
(\check{x}_\alpha, \check{y}_\alpha, \check{x}, \check{\theta}) := (x(q)^{\frac{1}{2}}x_\alpha, x(q)^{\frac{1}{2}}y_\alpha, x(q)^{-1}x, x(q)\theta).
\end{align}
Then, on the preimage of $B_1(q)$, uniformly as $x(q) \to 0^+$, the Riemannian metric associated with $\omega_h$ is $C^\infty$ equivalent to the standard Euclidean metric in the chart $(\check{x}_\alpha,\check{y}_\alpha,\check{x},\check{\theta})$. This is easy to check using Lemmas \ref{l:formula}--\ref{lem-n}, equation \eqref{eq:ballsqueeze}, and the fact that $\varphi(z')$ is a homogeneous quadratic polynomial.

We also introduce new holomorphic coordinates 
\begin{align}
\check{z}_\alpha := x(q)^{\frac{1}{2}}z_\alpha = \check{x}_\alpha + i\check{y}_\alpha,\;\,\check{w} := x(q)w = \frac{1}{2}\left(\varphi(\check{z}') - \frac{1}{\check{x}}\right) + i\check\theta.
\end{align}
Using indices with respect to $(\check{z}',\check{w})$, the Monge-Amp\`ere equation \eqref{eq-main} can be written as
\begin{align}\label{eq-guesswhat}
\int_0^1 (\omega_h + t \ddbar u)^{\bar{k} j} dt  \cdot \partial_j\partial_{\bar k} u = u
\end{align}
on a fixed ball in $\C^n$, where the coefficient matrix of $\omega_h$ is uniformly comparable to the identity to all orders. Since $\omega = \omega_h + \ddbar u$ is uniformly equivalent to $\omega_h$ and $|u|$ is in particular $O(1)$, the standard theory of the Monge-Amp\`ere equation now tells us that $u$ is bounded to all orders. Using this and the fact that $|u| = O(x)$, the lemma follows from Schauder theory applied to \eqref{eq-guesswhat}.
\end{proof}

By interpolation, we will now show that all derivatives of $u$ with at least one $x_\alpha,y_\alpha,\theta$-direction are actually $O(x^\infty)$ because all of these directions are small under our metric $\omega_h$. For this and for the rest of the paper, it is helpful to directly prove an even stronger estimate in the $\theta$-directions. We do this using an argument similar to \cite[Prop 2.9(i)]{H}, which does not rely on the flatness of $D$.

\begin{lemma}\label{lem:circle}
Assuming only that $|\nabla_{\omega_h}u|_{\omega_h} = O(x)$ as $x \to 0^+$ and that $\omega$ is uniformly equivalent to $\omega_h$,  we obtain an $\epsilon > 0$ such that $|\nabla_{\omega_h}^k (\mathcal{L}_{\partial_\theta} u)|_{\omega_h} = O_k(e^{-\epsilon/x})$ for all $k \in \N_0$.
\end{lemma}

\begin{proof}
Recall that $\partial_\theta$ is a globally defined holomorphic vector field with $\mathcal{L}_{\partial_\theta}\omega_h = 0$. Define $v = \mathcal{L}_{\partial_\theta} u$. Applying the operator $\mathcal{L}_{\partial_\theta}$ to the Monge-Amp\`ere equation \eqref{eq-main}, we thus obtain that
\begin{align}\label{eq-v}
\Delta_{\omega} v = v,
\end{align}
where $\omega = \omega_h + \ddbar u$ is by assumption uniformly equivalent to $\omega_h$.

The first step is to prove that $v \in H^1(U \setminus D, \omega_h)$. We have $|v| = O(x^2)$ because $|\nabla_{\omega_h} u|_{\omega_h} = O(x)$ and $|\partial_\theta|_{\omega_h} = O(x)$. Thus, $v \in L^2(U \setminus D, \omega_h)$ because ${\rm Vol}(U\setminus D,\omega_h) < \infty$. It is also not difficult to estimate the gradient of $v$ with respect to $\omega_h$. For example, this can easily be done using quasi-coordinates as in the proof of Lemma \ref{lem-v_higher}. The following method also works if $D$ is not flat. Assume that $\partial U = \{\rho = \rho_0\}$ for some $\rho_0 > 0$. For $k \in \N$, $k \geq \rho_0+1$, take a cutoff function $\eta_k = \eta_k(\rho)$ that equals $1$ for $\rho_0\leq \rho\leq k$ and $0$ for $\rho\geq k+1$, with $|\nabla_{\omega_h}\eta_k|_{\omega_h} \leq 2$. Then, by \eqref{eq-v},
 \begin{align}\label{eq:tru}
 \int_{U \setminus D} \eta_k^2 v\Delta_{\omega} v  \,dV_\omega = \int_{U \setminus D} \eta_k^2 v^2 dV_\omega.
 \end{align}
By integration by parts,
\begin{align}\label{eq:llala}
\int_{U \setminus D} \eta_k^2  v\Delta_{\omega} v \, dV_\omega  = - \int_{U \setminus D} \eta_k^2 |\nabla_{\omega} v|_\omega^2 \, dV_\omega -\int_{U \setminus D} 2\eta_k v \langle \nabla_{\omega} \eta_k, \nabla_{\omega} v\rangle_\omega \, dV_\omega.
\end{align}
After combining \eqref{eq:tru}--\eqref{eq:llala} and applying Cauchy-Schwarz in the obvious way, we can replace $\omega$ by $\omega_h$ because $\omega$ is uniformly equivalent to $\omega_h$. Letting $k \to \infty$, we obtain the desired $H^1$ bound.

For the main step of the proof, define $U_{\rho} = \{q\in {U \setminus D} : \rho(q) \geq  \rho\}$. Then, by \eqref{eq-v},
\begin{align}
E(\rho):= \int_{U_{\rho}}(|\nabla_\omega v|_{\omega}^2 + v^2) \, dV_{\omega} = -\int_{\partial U_{\rho}} v \mathbf{n}_{\omega}(v) \, dA_{\omega},
\end{align}
where $\mathbf{n}_\omega$ denotes the $\omega$-unit normal of $\partial U_\rho$ pointing towards the cuspidal end. Thus,
\begin{align}\label{eq:uhl}
E(\rho) \leq \left( \int_{\partial U_{\rho}} v^2 dA_{\omega}\right)^{\frac{1}{2}}\left( \int_{\partial U_{\rho}} (\mathbf{n}_\omega(v))^2 dA_{\omega} \right)^{\frac{1}{2}}.
\end{align}
We will now develop a Poincar\'e inequality to deal with the first factor.
 
By Lemma \ref{l:formula}, the restriction of $g_h$ to $Y := \partial U_\rho$ is given by
\begin{align}
g_h|_{Y} = 2(n+1)x(g_D + x \alpha^2),
\end{align}
where $\rho = -\sqrt{(n+1)/2} \log x$ as usual, $g_D$ is a fixed metric on $D$ pulled back under the $S^1$-orbit map $Y \to D$, and $\alpha := d\theta - \frac{1}{2}d^c\varphi$ is an $S^1$-connection form, i.e., $\alpha(\partial_\theta) = 1$ and $\mathcal{L}_{\partial_\theta} \alpha = 0$. The $S^1$-action preserves the metric $g := g_D + x\alpha^2$, so all eigenspaces of the Laplacian of $g$ are $S^1$-invariant. Thus, by the representation theory of $S^1$, every eigenspace $E = \{\phi \in L^2(Y,\R): \Delta_g\phi = -\lambda\phi\}$ can be written as the $L^2$-orthogonal sum of a space $E_0$ of $S^1$-invariant functions and of planes $\mathbb{R}\phi_1 + \mathbb{R}\phi_2$, where $\phi_1,\phi_2$ are $L^2$-orthonormal and $\mathcal{L}_{\partial_\theta}\phi_1 = -k\phi_2$, $\mathcal{L}_{\partial_\theta}\phi_2 = k\phi_1$ for some $k \in \N_{\geq 1}$. Then, for all $\phi \in E$,
\begin{align}\label{eq:uhl3}
\lambda \int_Y \phi^2 = \int_Y |\nabla \phi|^2 \geq \int_Y (\mathcal{L}_{x^{-1/2}\partial_\theta}\phi)^2 = x^{-1}\int_Y (\mathcal{L}_{\partial_\theta}\phi)^2 \geq x^{-1}\int_Y ({\rm proj}_{E_0^\perp}\phi)^2,\end{align}
where all metric quantities are understood in terms of $g$. The first inequality holds because $x^{-1/2}\partial_\theta$ is a $g$-unit vector, and the second inequality follows by splitting $\phi$ into its components in $E_0$ and in the planes of frequency $k \in \mathbb{N}_{\geq 1}$ described above. See \cite[pp.447--448]{KK} for a related argument.

In our situation, $v = \mathcal{L}_{\partial_\theta}u$ is $L^2$-orthogonal to all $S^1$-invariant functions on $Y$. Thus, if $\phi = {\rm proj}_E v$, then $\phi = {\rm proj}_{E_0^\perp}\phi$, so either $\phi = 0$ or $\lambda \geq x^{-1}$ by \eqref{eq:uhl3}. Thus, because $g_h|_{Y} = 2(n+1)xg$, 
\begin{align}\label{eq:uhl2}
\int_{\partial U_\rho} v^2 dA_{\omega_h} \leq 2(n+1)x^2 \int_{\partial U_\rho} (|\nabla_{\omega_h}v |_{\omega_h}^2 - (\mathbf{n}_{\omega_h}(v))^2)\, dA_{\omega_h}.
\end{align}
We now drop the $(\mathbf{n}_{\omega_h}(v))^2$ term from \eqref{eq:uhl2}, replace $\omega_h$ by $\omega$ as the reference metric, plug the resulting inequality into \eqref{eq:uhl}, and trivially bound the $(\mathbf{n}_\omega (v))^2$ term in \eqref{eq:uhl} above by $|\nabla_\omega v|_\omega^2$. Thus,
\begin{align}
E(\rho) \leq Cx \int_{\partial U_\rho} |\nabla_\omega v|_\omega^2\,dA_\omega \leq -  Cx \left(\sup_{\partial U_\rho} |\nabla_\omega \rho|_{\omega}\right) E'(\rho) \leq -Cx E'(\rho)\label{eq-ODE-Ev}
\end{align}
for some uniform constant $C$, where the second inequality follows from the coarea formula. Dividing both sides of $\eqref{eq-ODE-Ev}$ by $E(\rho)$ and integrating from $\rho_0$ to $\rho$, we derive that for some $\epsilon > 0$,
\begin{align}
E(\rho) \leq E(\rho_0) \exp\left(-\frac{1}{C}\int_{\rho_0}^\rho \frac{d\rho}{x(\rho)} \right) \leq C e^{-\epsilon/x}.
\end{align}

This decay rate is so fast that it now suffices to use elliptic regularity at the scale of the injectivity radius. Indeed, Moser iteration applied to \eqref{eq-v} on $\omega_h$-balls of radius $x$ yields a pointwise bound for $v$ which is $O(x^{-N}e^{-\epsilon/x})$ for some $N > 0$, and $O(x^{-N}e^{-\epsilon/x}) = O(e^{-(\epsilon/2)/x})$. Then higher-order estimates for $v$ follow from Schauder theory applied to \eqref{eq-v} on balls of radius $x$. Here we also need higher-order estimates for the coefficient metric $\omega$, which follow from the standard regularity theory of the Monge- Amp\`ere equation on $\omega_h$-balls of radius $x$. Indeed, by assumption, $|\nabla_{\omega_h} u|_{\omega_h}$ decays exponentially fast in terms of the distance function $\rho$, so $u$ converges to a constant exponentially fast in terms of $\rho$, and $\omega$ is uniformly equivalent to $\omega_h$, so $|\nabla^k_{\omega_h} u|_{\omega_h} = O_k(x^{-k})$ from the standard theory.
\end{proof}

Finally, we obtain the following self-improvement of Lemma \ref{lem-v_higher}. This is similar to \cite[Lemma 2.5]{JiangShi}, although in \cite{JiangShi} this improvement only applies to the $\theta$-derivatives because the $x_\alpha,y_\alpha$-directions have length $O(1)$, and here we already have a better estimate for the $\theta$-derivatives.

\begin{lemma}\label{lem-infinity}
Assuming the conclusions of Lemmas \ref{lem-v_higher} and \ref{lem:circle}, we obtain that $u = \mathbf{O}(x)$.
\end{lemma}

The $\mathbf{O}$ notation is as in Definition \ref{defn-boldO}. For this we need to choose a finite cover of $U \setminus D$ by charts $(x_\alpha,y_\alpha,x,\theta)$, and we will use the ones constructed in the proof of Lemma \ref{lem-v_higher} using the flatness of $D$. It again seems likely that, using \cite{HeinTosatti}, the proof can be made to work even if $D$ is not flat.

\begin{proof}
Fix any point $q \in U \setminus D$. Using the rescaled chart $(\check{x}_\alpha,\check{y}_\alpha,\check{x},\check{\theta})$ of \eqref{eq-quasi-coor} on the preimage of the set $\{c_n^{-1}x(q) < x < c_nx(q)\}\supset B_1(q)$, which makes the pullback of $\omega_h$ Euclidean to all orders uniformly as $x(q) \to 0^+$, we can state the results of Lemmas \ref{lem-v_higher}--\ref{lem:circle} as follows: for all $a,b,c,d \in \N_0$, 
\begin{align}
|\partial_{\check{x}_\alpha}^a\partial_{\check{y}_\alpha}^b \partial_{\check{x}}^c \partial_{\check{\theta}}^d u| = \begin{cases}O_{a,b,c}(x(q)) &\text{if}\;d=0\\O_{a,b,c,d}(x(q)^{-d}e^{-\epsilon/x(q)})&\text{if}\;d \geq 1
\end{cases} \quad\text{as}\;x(q)\to 0^+.
\end{align}
Converting these estimates back to the unrescaled chart $(x_\alpha,y_\alpha,x,\theta)$, we gain a factor of $x(q)^{\frac{a+b}{2} - c + d}$ on the right-hand side. This implies property \eqref{eq:fatO1} for $k = 1$ as well as property \eqref{eq:fatO2} for $d \geq 1$. Thus, it remains to verify \eqref{eq:fatO2} for $d = 0$ and $a + b \geq 1$. For this it suffices to prove the implication
\begin{align}\label{eq:implic}
\begin{split}
&\;\forall A,B,C,D\in \N_0: |\partial_{x_\alpha}^{A}\partial_{y_\alpha}^{B} \partial_x^C \partial_\theta^D u| = O_{A,B,C}(x^{\frac{A+B}{2}-C})\cdot \begin{cases}O(x)&\text{if}\;D=0\\O_D(e^{-\epsilon/x}) &\text{if}\;D \geq 1\end{cases}\\
\Longrightarrow &\; \forall a,b,c \in \N_0: |\partial_{x_\alpha}^a\partial_{y_\alpha}^b\partial_x^c u| = O_{a,b,c}(x^\infty)\;\,\text{if}\;a+b\geq 1.
\end{split}
\end{align}

To prove \eqref{eq:implic}, first note that $v := \partial_{x_\alpha}^a\partial_{y_\alpha}^b\partial_x^c u$ satisfies $|\partial_\theta v| = O(e^{-\epsilon'/x})$ for any $0 < \epsilon' < \epsilon$ thanks to the left-hand side of \eqref{eq:implic}. Thus, if $\bar{v}$ denotes the average of $v$ with respect to $\theta$ over $[0,2\pi]$, then $|v - \bar{v}| = O(e^{-\epsilon'/x})$, so it is actually enough to prove that $|\bar{v}| = O(x^\infty)$. Let $\bar{u}$ denote the average of $u$ with respect to $\theta$ over $[0,2\pi]$, so that $\bar{v} = \partial_{x_\alpha}^a \partial_{y_\alpha}^b \partial_x^c \bar{u}$. Fixing the $x$-variable, we can view $\partial_x^c\bar{u}$ as a function on the flat torus $D$. Similarly, we can view $\partial_{x_\alpha},\partial_{y_\alpha}$ as commuting Killing vector fields on $D$. Because $a + b \geq 1$ and $D$ is a compact manifold without boundary, it follows that the average of $\bar{v}$ over $D$ is zero, so $\bar{v}(o) = 0$ for some point $o \in D$. Without loss, $o$ is the image of $0$ under the covering map $\C^{n-1} \to D$. Then, for a fixed convex fundamental domain $\Omega \ni 0$ of the deck group action,
\begin{align}
\sup_{z'\in\Omega}|\bar{v}(z')| = \sup_{z'\in\Omega}\left|\int_0^1 \frac{d}{dt}\bar{v}(tz')\, dt\right| \leq {\rm diam}(\Omega)\left(\sup_\Omega |\partial_{x_\alpha}\bar{v}| + \sup_\Omega |\partial_{y_\alpha}\bar{v}|\right).
\end{align}
Iterating this argument $k$ times and applying the left-hand side of \eqref{eq:implic} with $A+B = a+b+k$ (and $C = c$, $D = 0$), we obtain that $|\bar{v}| = O_k(x^{k_0+(k/2)})$ for all $k \in \N_0$, where $k_0 := \frac{a+b}{2}-c+1$.
\end{proof}

\subsection{Proof of Theorem \ref{thm:expansion-powers}}\label{ss:proof-3.1}

Throughout this section, assuming the conclusion of Lemma \ref{lem-infinity}, $D$ can again be an arbitrary compact Calabi-Yau manifold rather than a flat complex torus.

The following definition is modeled after \cite[Defn 5.1]{JiangShi}.

\begin{definition}\label{def-exp}
After shrinking $U$ if necessary, consider the finite cover of $U \setminus D$ by local charts of the form $(x_\alpha,y_\alpha,x,\theta)$ that was used in Lemma \ref{lem-infinity} via Definition \ref{defn-boldO}. For any $k\in \mathbb{N}_0$,
we say that a function $f \in C^\infty(U\setminus D)$ has an expansion of order $x^k$ if, under all charts of this cover,
\begin{align}
f = \sum_{i=0}^k C_i x^i + R_k,
\end{align}
where the $C_i$ are constants and there exists an $\epsilon>0$ such that $R_k = \mathbf{O}(x^{k+\epsilon})$.
\end{definition}

Thus, Lemma \ref{lem-infinity} implies that $u$ has an expansion of order $x^0$ with $C_0  = 0$ and with $\epsilon=1$.

\begin{lemma}\label{lem-PDE-ODE}
If $\psi$ is a $C^2$ function which only depends on $x$, then 
\begin{align}\label{eq-psi}
\frac{(\omega_h + \ddbar \psi)^n}{\omega_h^n} =  \frac{(n+1 + x\psi')^{n-1}(n+1 + x(x\psi)'')}{(n+1)^{n}}.
\end{align}
\end{lemma}

\begin{proof}
Consider the holomorphic chart $(z',z_n)$ underlying the given smooth chart $(x_\alpha,y_\alpha,x,\theta)$ as in Section \ref{ss:modelmetric}. Then, by \eqref{eq-gij}, the metric coefficients of $\omega_h$ are given by
\begin{align}\begin{split}\label{eq-g-detail}
g_{\alpha \bar \beta} & =(n+1)(-x\varphi_{\alpha\bar \beta}+x^2\varphi_\alpha \varphi_{\bar \beta}),\\
g_{\alpha \bar n} & =-  (n+1)\frac{x^2}{\bar{z}_n}\varphi_\alpha,\\
g_{n \bar n} & =(n+1)\frac{x^2}{|z_n|^2}.
\end{split}
\end{align}
On the other hand, using \eqref{eq-partial} we compute
\begin{align}\begin{split}
\psi_{\alpha \bar \beta}& = x^3(2\psi' + x \psi'')\varphi_\alpha \varphi_{\bar \beta} - x^2  \psi' \varphi_{\alpha \bar \beta},\\
\psi_{\alpha \bar n} &= -\frac{x^3}{\bar{z}_n}(2\psi' + x \psi'')\varphi_\alpha,\\
\psi_{n \bar n} &= \frac{x^3}{|z_n|^2}(2\psi' + x \psi'').
\end{split}
\end{align}
With this, we can now compute the ratio of $\det(g_{j\bar k} + \psi_{j\bar k})$ and $\det(g_{j\bar k})$. In both of these determinants, we can take out common factors of $\bar{z}_n^{-1}$ in the $n$-th column and $z_n^{-1}$ in the $n$-th row. Then, by adding a $\varphi_{\bar \beta}$ multiple of the $n$-th column of each matrix to its $\beta$-th column, we create a block triangular matrix in each case. A simple calculation concludes the proof of the lemma.
\end{proof}

We now rewrite the complex Monge-Amp\`ere equation \eqref{eq-main} as
\begin{align}\label{eq:MAOp}
M_h(u) :=\log\left(\frac{(\omega_h + \ddbar u)^n}{\omega_h^n}\right)-u =0,
\end{align}
and we decompose $M_h$ into its linearization at $\omega_h$ and its nonlinear parts via
\begin{align}
M_h = L_h + Q_h.
\end{align}
Here, $L_h$ was already defined in \eqref{eq:defineLh}, and $Q_h$ is a convergent power series in $\ddbar u$ without constant or linear terms. Using Lemma \ref{lem:Lu}, we can further rewrite this as 
\begin{align}\label{eq-ODE}
x^2 u_{xx} +(n+1) xu_x - (n+1) u = G
\end{align}
in every local chart of our chosen cover of $U \setminus D$, where
\begin{align}\label{eq-F1}
\begin{split}
G := x^{-1}\varphi^{\bar \beta \alpha} u_{\alpha \bar \beta} -(2x)^{-2}u_{\theta\theta} - (n+1)(F + Q_h(u)), \\
F =F(x_\alpha,y_\alpha, x^{-1} u_{\theta\alpha}, x^{-1} u_{\theta\bar\alpha},x^{-1} u_{\theta\theta}),
\end{split}
\end{align}
and $F$ is smooth in $x_\alpha,y_\alpha$ and linear homogeneous in the remaining arguments.

We have the following lemma, which is analogous to \cite[Lemma 5.1]{JiangShi}.

\begin{lemma}\label{lem-UF}
For all $k \in \N_0$, if $u$ has an expansion of order $x^k$, with $\epsilon>\frac{1}{2}$ if $k = 0$, then $G$ has an expansion of order $x^{k+1}$.
\end{lemma}

\begin{proof}
Those contributions to $G$ that are linear in $u$ can be written in terms of tangential derivatives of $u$. Thus, by Lemma \ref{lem-infinity}, they are $\mathbf{O}(x^\infty)$, i.e., $\mathbf{O}(x^N)$ for any given $N \in \N$. In particular, they have an expansion to any given order, and the polynomial part of this expansion is zero.

It remains to discuss the term $Q_h(u)$. For this we need to use our assumption that $u = \psi_k + R_k$, where $\psi_k$ is a degree $k$ polynomial in $x$ and the remainder $R_k$ is $\mathbf{O}(x^{k+\epsilon})$. We first observe that 
\begin{align}\label{eq:redentor}
R_k = \rho_k(x) + \mathbf{O}(x^\infty),
\end{align}
 where $\rho_k(x) = \mathbf{O}(x^{k+\epsilon})$ is a smooth radial function. Indeed, for $0<t\ll 1$ we can simply define $\rho_k(t)$ to be the average of $R_k$ over the level set $\{x = t\}$ with respect to the restriction of the Riemannian metric associated with $\omega_h$. Then the properties $\rho_k(x) = \mathbf{O}(x^{k+\epsilon})$ and \eqref{eq:redentor} follow from the fact that $R_k = \mathbf{O}(x^{k+\epsilon})$, which implies in particular that all tangential derivatives of $R_k$ are $\mathbf{O}(x^\infty)$.

Now, also by definition,
\begin{align}
Q_h(u) = \log\left(\frac{(\omega_h + \ddbar u)^n}{\omega_h^n}\right) - \frac{d}{dt}\biggr|_{t=0} \log \left( \frac{(\omega_h + t\ddbar u)^n}{\omega_h^n} \right).
\end{align}
Writing $u = \psi_k + \rho_k + \mathbf{O}(x^\infty)$ in this formula, using Lemma \ref{lem-PDE-ODE}, and defining
\begin{align}\label{eq:quadraticterms}
\mathbf{q}(t) := \log\left(1+\frac{t}{n+1}\right) - \frac{t}{n+1},
\end{align}
we obtain that
\begin{align}\label{eq:cristo}
Q_h(u) = (n-1)\mathbf{q}(x\psi_k'+x\rho_k') + \mathbf{q}(x(x\psi_k)'' + x(x\rho_k)'') + \mathbf{O}(x^\infty).
\end{align}
From this we can now deduce that $Q_h(u)$ has an expansion of order $x^{k+1}$.

Indeed, $\mathbf{q}$ is a convergent power series without constant or linear terms. Thus, if we expand \eqref{eq:cristo} as a convergent power series in the four variables $a = x\psi_k'$, $b = x\rho_k'$, $c = x(x\psi_k)''$ and $d = x(x\rho_k)''$, then each monomial will be divisible by $a^2$, $ab$, $b^2$, $c^2$, $cd$ or $d^2$. Now observe that $a,c$ are polynomials in $x$ without constant terms while $b,d$ are smooth functions of $x$ which are $\mathbf{O}(x^{k+\epsilon})$. Thus, all monomials in $a,b,c,d$ that appear in the expansion of \eqref{eq:cristo} are either polynomials in $x$ or $\mathbf{O}(x^{\min\{1+k+\epsilon,2k+2\epsilon\}})$ smooth functions of $x$. Because $\epsilon > \frac{1}{2}$ if $k = 0$ and because the sum of these monomials converges as a power series in $a,b,c,d$, it follows that $Q_h(u)$ has an expansion of order $x^{k+1}$, as desired.
\end{proof}

Since we know from Lemma \ref{lem-infinity} that $u$ has an expansion of order $x^0$ with $\epsilon = 1$, Lemma \ref{lem-UF} tells us that $G$ has an expansion of order $x^1$. More precisely,
$G = \mathbf{O}(x^2)$ by the proof of Lemma \ref{lem-UF}.

The following lemma is analogous to but much simpler than \cite[Proposition 5.1]{JiangShi}.

\begin{lemma}\label{lem-FU}
Let $\ell \in\N_{\geq 1}$.
If $G$ has an expansion of order $x^{\ell}$, then $u$ has an expansion of order $x^\ell$. 
\end{lemma}

\begin{proof}
Fix any chart. Fix an arbitrary point $(x_\alpha,y_\alpha,\theta) \in D \times S^1$ and let $0 < x_0 \ll 1$. View \eqref{eq-ODE} as a second-order ODE for $u(x_\alpha,y_\alpha,x,\theta)$ with initial values at $x = x_0$. The corresponding homogeneous ODE has the two linearly independent solutions $x$ and $x^{-n-1}$. Using variation of parameters and the fact that $u = O(1)$ and $G = O(x^{1+\epsilon})$ for some $\epsilon>0$, we derive that
\begin{align}\begin{split}
\label{eq-Sln-lambda0}
u(x) = \left( \frac{u(x_0)}{x_0} +\frac{x_0^{-n-2}}{n+2} \int_0^{x_0} t^{n} G(t) \,dt \right) x - \frac{x^{-n-1}}{n+2}  \int_0^{x} t^{n} G(t) \,dt -  \frac{x}{n+2} \int_x^{x_0} t^{-2} G(t)\,dt.
\end{split}
\end{align}
Evaluation at $(x_\alpha,y_\alpha,\theta)$ is implicit in this formula. 

If $G$ is a degree $\ell$ polynomial in $x$ without constant or linear terms, then \eqref{eq-Sln-lambda0} implies that $u$ is a degree $\ell$ polynomial in $x$ without constant terms.

If $G = \mathbf{O}(x^{\ell+\epsilon})$, then we can deduce from \eqref{eq-Sln-lambda0} that
\begin{align}\label{eq:step666}
u(x) = \left(\frac{u(x_0)}{x_0} +  \frac{x_0^{-n-2}}{n+2}\int_0^{x_0} t^n G(t) \, dt - \frac{1}{n+2}\int_0^{x_0} t^{-2}G(t)\,dt \right)x + \mathbf{O}(x^{\ell + \epsilon})
\end{align}
by splitting the last integral in \eqref{eq-Sln-lambda0} into an integral from $0$ to $x_0$ and an integral from $0$ to $x$. The term in parentheses in \eqref{eq:step666} is by definition independent of $x$, and all of its tangential derivatives are $\mathbf{O}(x^\infty)$ by Lemma \ref{lem-infinity}. Thus, this term is constant, and so $u = const \cdot x + \mathbf{O}(x^{\ell+\epsilon})$.
\end{proof}

Iterating Lemmas \ref{lem-UF}--\ref{lem-FU}, we obtain that $u$ has an expansion to any order, i.e., for all $k \in \N_0$ there exists a degree $k$ polynomial $\psi_k(x)$ without a constant term such that
$u = \psi_k(x) +  \mathbf{O}(x^{k+\epsilon})$ for some $\epsilon>0$. In fact, the above arguments show that this holds for $\epsilon=1$. Moreover, the polynomials $\psi_{k}(x)$ obviously stabilize in the sense that $\psi_{k}(x) = \psi_{k-1}(x) + C_{k}x^{k}$ for all $k \in \N_{\geq 1}$. Thus, we may pass to a limit $\psi = \lim_{k} \psi_k$ in the ring of all formal power series in $x$ without constant terms.

By construction, $\psi$ is a formal power series solution to the complex Monge-Amp\`ere equation \eqref{eq-main}. Equivalently, by \eqref{eq-ODE}--\eqref{eq-F1} and \eqref{eq:cristo}, $\psi$ formally solves the equation
\begin{align}\label{eq:myfavorite}
x^2 \psi'' +(n+1) x\psi' - (n+1)\psi = -(n+1)[(n-1)\mathbf{q}(x\psi') + \mathbf{q}(x(x\psi)'')],
\end{align}
where $\mathbf{q}(t)$ as in \eqref{eq:quadraticterms} is a convergent power series in $t$ without constant or linear terms. Any formal solution $\psi(x) = C_1 x + C_2 x^2 + \cdots$ to \eqref{eq:myfavorite} is uniquely determined by its leading coefficient $C_1$. Indeed, comparing terms of order $x^k$ in \eqref{eq:myfavorite} yields a coefficient $(k-1)(k+n+1)C_k$ on the left side and some polynomial in $C_1, \ldots, C_{k-1}$ on the right. For $k = 1$ this is a tautology,  and for $k \geq 2$ we can solve this equation uniquely for $C_k$ in terms of $C_1,\ldots,C_{k-1}$. However, instead of actually solving this system of equations, it is simpler to observe that if $C_1 = -(n+1)c$, then the function $-(n+1) \log (1+ cx )$ has the right leading term and solves \eqref{eq:myfavorite}, so that $\psi(x) = -(n+1)\log(1+cx)$. 

Theorem \ref{thm:expansion-powers} has been proved.

\section{Improvement of the remainder}\label{sec:exp-decay}

In this section, we complete the proof of the Main Theorem by improving the remainder $O(x^\infty)$ of Theorem \ref{thm:expansion-powers} to the decay rate of the next generation of solutions of the homogeneous linearized PDE. This is done using a continuity argument. For this we first construct a right inverse to the model linear operator, which leads to a useful representation formula (Proposition \ref{prop:rep}). As in Section \ref{sec:expansion-powers}, the proof of this uses ODE arguments, although unlike in Section \ref{sec:expansion-powers} we now need to work globally on the cross- section and consider the spectrum of the Laplacian (but it suffices to use $S^1$-invariant eigenfunctions). Using this formula, we then carry out the continuity argument in Proposition \ref{prop:key-prop}.

\subsection{A representation formula}\label{sec-sln-ODE}

Assume the same setup as in the Main Theorem. By Theorem \ref{thm:expansion-powers}, there exists a $c \in \R$ such that if we define $\psi(x) := -(n+1)\log(1+cx)$, then the function $u^* := u - \psi$ satisfies $u^* = O(x^\infty)$ along with all of its $\omega_h$-derivatives. Recall that we had originally defined $\psi$ to be the unique formal power series solution to $(\omega_h + i\partial\overline{\partial}\psi)^n = e^{\psi}\omega_h^n$ whose leading coefficient agrees with the Green's function coefficient of an asymptotic expansion of $u$. Using $\omega_h^* := \omega_h + \ddbar\psi$ as a reference metric instead of $\omega_h$, we can therefore rewrite our equation as
\begin{align}\label{eq-main-star}
(\omega_h^* + i\partial\overline{\partial} u^*)^n = e^{u^*} (\omega_h^*)^n.
\end{align}
A crucial observation at this point is that $\omega_h^* = \omega_{h^*}$ for $h^* = e^{-c}h$. The radial coordinate associated with this new Hermitian metric is $x^* = (x^{-1}+c)^{-1}$. Passing from $\omega_h, x$ to $\omega_{h^*}, x^*$ makes no difference to the statement of the Main Theorem. Thus, we can assume without loss that $c = 0$. In fact:

\begin{reduction}\label{thm-main-red}
To prove the Main Theorem, it is enough to prove that 
\begin{equation}\label{eq:thm-main-red}
\text{if}\;\,(\omega_h + \ddbar u)^n = e^u \omega_h^n\;\,\text{and}\;\,|u| = O(x^\infty),\,\text{then}\;\,|u| = O\left(x^{-\frac{n}{2}+\frac{1}{4}}e^{-\frac{2\sqrt{\lambda_1}}{\sqrt{x}}}\right).
\end{equation}
\end{reduction}

The desired derivative decay follows from standard elliptic regularity applied on the universal covers of $\omega_h$-geodesic balls of radius $\epsilon\sqrt{t}$ centered at points of height $x = t$ for some small but fixed $\epsilon>0$. The point is that $\sqrt{t}$ is the biggest order of magnitude such that functions of the form $e^{-\delta/\sqrt{x}}$ vary by bounded factors over such balls. At the same time, these balls have fundamental group $\mathbb{Z}$, and balls of radius $\epsilon\sqrt{t}$ in their universal covers satisfy uniform geometry bounds at their own scale.

We now go back to the ODE version \eqref{eq-ODE} of the Monge-Amp\`ere equation, i.e., 
\begin{align}\label{eq-ODE-19}
\begin{split}
x^2 u_{xx} +(n+1) xu_x - (n+1) u = G,\\
G = x^{-1}\varphi^{\bar \beta \alpha} u_{\alpha \bar \beta} -(2x)^{-2} u_{\theta\theta} - (n+1)(F + Q_h(u)),
\end{split}
\end{align}
where $F$ is smooth in $x_\alpha,y_\alpha$ and linear homogeneous in $x^{-1} u_{\theta\alpha}, x^{-1} u_{\theta\bar\alpha},x^{-1} u_{\theta\theta}$ and where $Q_h(u)$ is a power series in $\ddbar u$ without constant or linear terms. We move the Laplacian in the $D$-directions out of the inhomogeneity, $G$, and average with respect to $\theta \in [0,2\pi]$. This yields
\begin{align}\label{eq-PDE}
\begin{split}
x^2\mathring{u}_{xx} +(n+1) x\mathring{u}_x -(n+1) \mathring{u}  -x^{-1}\varphi^{\bar \beta \alpha} \mathring{u}_{\alpha \bar \beta}= \mathring{G},\\
\mathring{u} :=\frac{1}{2\pi}\int_{0}^{2\pi} u\, d\theta, \;\,\mathring{G} := -(n+1)\frac{1}{2\pi}\int_{0}^{2\pi} Q_h(u) \, d\theta.
\end{split}
\end{align}
Crucially, $\mathring{u}$ and $\mathring{G}$ are globally well-defined on $(0,\delta] \times D$ for some $\delta>0$, not just chart by chart. This is because  the operator $Q_h$ is global and, for all global functions $v$, the integral $\int_0^{2\pi} v\,d\theta$ in each chart can be identified with the $S^1$-invariant global function $\int_{S^1} g^*v\,dg$, where $(g^*v)(p) = v(g \cdot p)$ (using the given global $S^1$-action) and where $dg$ denotes the standard measure on $S^1$. Another key point is that $|u - \mathring{u}| = O(e^{-\epsilon/x})$ by Lemma \ref{lem:circle}, so $u$ satisfies the estimates in \eqref{eq:thm-main-red} if and only if $\mathring{u}$ does.

Let $\lambda_\ell, \varphi_\ell$ denote the eigenvalues and $L^2$-orthonormalized eigenfunctions of the operator $\varphi^{\bar \beta \alpha} \partial_\alpha \partial_{\bar \beta}$ on $D$, where $\ell \in \N_0$ and $0 = \lambda_0 < \lambda_1 \leq \lambda_2 \leq \ldots$ We will use the Fourier decompositions
\begin{align}\label{eq:fourier}
\mathring{u} = \sum_{\ell=0}^\infty \mathring{u}_\ell(x)\varphi_\ell, \;\, \mathring{G} = \sum_{\ell=0}^\infty \mathring{G}_{\ell}(x) \varphi_{\ell},
\end{align}
which converge in $C^\infty_{\rm loc}$ as usual. Then \eqref{eq-PDE} decomposes into the ODEs
\begin{align}
x^2 \mathring{u}_{\ell}''(x)+(n+1)x \mathring{u}_{\ell}'(x)-(n+1) \mathring{u}_{\ell}(x) - \lambda_\ell x^{-1} \mathring{u}_{\ell}(x) = \mathring{G}_{\ell}(x).
\end{align}
We already treated the case $\ell = 0$ in Section \ref{sec:expansion-powers}, although in Section \ref{sec:expansion-powers} we were able to work pointwise on the cross-section with no need to integrate in any direction.

Denote $\alpha = n+2$. For all $\ell \in \N_{\geq 1}$ denote
\begin{align}\label{eq-H1l-H2l}
H_{1,\ell}(x) = {\frac {{1}}{x^{n/2}}{{\sl I}_{\alpha}\left(2
\,  {\frac {\sqrt {\lambda_{\ell}}}{\sqrt {x}}}\right)}}, \,\, H_{2,\ell}(x) = {\frac {{ 1}}{x^{n/2}}{{\sl K}_{\alpha}\left(2\,{\frac {\sqrt {\lambda_{\ell}}}{\sqrt {x}}}\right)}},
\end{align}
where $I_\alpha(s)$ is the modified Bessel function of the first kind and $K_\alpha(s)$ is the modified Bessel function of the second kind. It is known that as $s \to \infty$ there exist asymptotic expansions
\begin{align}
I_\alpha(s) &=\frac{e^s}{\sqrt{2\pi s}} \left(1 -\frac{4\alpha^2-1}{8 s} +  \frac{(4\alpha^2-1)(4\alpha^2-9)}{2!  (8 s)^2} - \frac{(4\alpha^2-1)(4\alpha^2-9)(4\alpha^2-25)}{3!  (8 s)^3}+\cdots   \right),\label{eq:bessel1}\\
K_\alpha(s) &=\sqrt{\frac{\pi}{2s}}e^{-s} \left(1 +\frac{4\alpha^2-1}{8 s} +  \frac{(4\alpha^2-1)(4\alpha^2-9)}{2!  (8 s)^2} + \frac{(4\alpha^2-1)(4\alpha^2-9)(4\alpha^2-25)}{3!  (8 s)^3}+\cdots   \right).\label{eq:bessel2}
\end{align}
In particular, if $\alpha = n+2$ is fixed, then for all $\epsilon \in (0,1)$ there exists a universal $s_\epsilon > 0$ depending only on $\epsilon$ such that the two factors in parentheses in \eqref{eq:bessel1}--\eqref{eq:bessel2} are in the interval $[1-\epsilon,1+\epsilon]$ for $s \geq s_\epsilon$. This remark is useful because it provides uniform estimates for $H_{1,\ell}(x)$ and $H_{2,\ell}(x)$ as $x \to 0^+$, where the constants depend on $\ell$ 
but the estimates hold for all $x \in (0,\delta]$ with $\delta$ independent of $\ell$. 

We are now able to state the main result of this section. A similar representation formula was used in a different setting in \cite[Section 6]{KK}. We will say more about this at the end of the paper. 

\begin{proposition}\label{prop:rep}
Assume only that $u = \mathbf{O}(x^M)$ as $x \to 0^+$ for some $M > 1$ rather than $|u| = {O}(x^\infty)$. Denote the restriction of a function $v$ on the total space to the slice $\{x = t\}$ by $v(t)$. Fix $0 < x_0 \ll 1$. Then, with all of the infinite sums converging in $C^0_{\rm loc}$, it holds on $(0,x_0] \times D$ that
\begin{align}\begin{split}\label{eq-Soln-Bessel-3}
 \mathring{u}(x) &=
- \frac{x^{-n-1}}{n+2}  \int_0^{x} t^{n} \mathring{G}_0(t)\, dt +  \frac{x}{n+2} \int_0^{x} t^{-2} \mathring{G}_0(t)\,dt\\
&+ \sum_{\ell=1}^\infty  \left( \frac{u_{\ell}(x_0)}{H_{2,\ell}(x_0)} +\frac{2H_{1,\ell}(x_0)}{H_{2,\ell}(x_0)} \int_0^{x_0} t^{n-1} H_{2,\ell}(t) \mathring{G}_{\ell}(t)\,  dt \right) H_{2,\ell}(x)\varphi_\ell\\
&-  \sum_{\ell=1}^\infty\left( 2H_{1,\ell}(x) \int_0^{x} t^{n-1}H_{2,\ell}(t) \mathring{G}_{\ell}(t)\, dt \right)\varphi_\ell -  \sum_{\ell=1}^\infty \left( 2H_{2,\ell}(x) \int_x^{x_0} t^{n-1} H_{1,\ell}(t) \mathring{G}_{\ell}(t)\,
dt \right)\varphi_\ell.
\end{split}
\end{align}
\end{proposition}

\begin{proof}
Consider an ODE of the general form
\begin{align}\label{eq-u}
x^2 v''(x) + (n+1) x v'(x) -(n+1)v(x) - \lambda x^{-1}  v(x)=f(x)
\end{align}
for some constant $\lambda>0$.
A function of the form $v(x) = x^{-n/2} F(2\sqrt{\lambda/x})$ is a homogeneous solution of this ODE if and only if $F(s)$ satisfies the modified Bessel equation
\begin{align}
s^2 F''(s) + s F'(s) -(\alpha^2+s^2 )F(s) = 0
\end{align}
for $\alpha =  n+2$.
Thus, the two homogeneous solutions of \eqref{eq-u}  are
\begin{align}
H_1(x) = {\frac {{1}}{x^{n/2}}{{\sl I}_{\alpha}\left(2
\,  {\frac {\sqrt {\lambda}}{\sqrt {x}}}\right)}}, \,\, H_2(x) = {\frac {{ 1}}{x^{n/2}}{{\sl K}_{\alpha}\left(2\,{\frac {\sqrt {\lambda}}{\sqrt {x}}}\right)}}
\end{align}
as above.
By Abel's formula and by \eqref{eq:bessel1}--\eqref{eq:bessel2},
\begin{align}\label{eq-abel}
I_\alpha(s) K_\alpha'(s) - I_\alpha'(s) K_\alpha(s) = -\frac{1}{s}.
\end{align}
By variation of parameters, the general solution to \eqref{eq-u} is 
\begin{align}\label{eq:general-soln}
v(x) = c_1 H_1(x) + c_2 H_2(x) &+ H_1(x) \int_x^{x_0} \frac{H_2(t) f(t)}{t^2 W(t)}\, dt -H_2(x) \int_x^{x_0} \frac{H_1(t) f(t)}{t^2W(t)} \, dt,
\end{align}
where thanks to \eqref{eq-abel} the Wronskian $W(t)$ is given by
\begin{align}\label{eq:wronski}
W(t) = H_1(t) H_2'(t)-H_1'(t) H_2(t) = -\sqrt{\lambda}t^{-n-\frac{3}{2}} (I_\alpha K_\alpha'- I_\alpha' K_\alpha)(2\sqrt{\lambda/t}) = \frac{1}{2t^{n+1}}.
\end{align}
Thus, \eqref{eq:general-soln} simplifies to
\begin{align}
\label{eq-Sln-u0}
v(x) = c_1 H_1(x) + c_2 H_2(x) &+ 2H_1(x) \int_x^{x_0}t^{n-1} H_2(t) f(t)\, dt - 2H_2(x)  \int_x^{x_0} t^{n-1} H_1(t) f(t) \, dt.
\end{align}
To solve out the constants $c_1, c_2$, we first take $x=x_0$, so that
\begin{align}
v(x_0) =c_1 H_1(x_0) + c_2 H_2(x_0).
\end{align}
Next we divide  \eqref{eq-Sln-u0} by $H_1(x)$ and let $x \rightarrow 0^+$. Notice that $H_1(x) \rightarrow \infty$ and $H_2(x) \rightarrow  0$ as $x\rightarrow 0^+$. We claim that if $|v| = O(1)$ and $|f| = O(1)$, then it follows that
\begin{align}\label{eq:thelimitIwant}
0 = c_1 + 2 \int_0^{x_0} t^{n-1}H_2(t) f(t)\, dt,
\end{align}
so we can solve for $c_1$ and $c_2$. This is clear except for the limit of the second integral term from \eqref{eq-Sln-u0}. To see that this is zero, we can assume that $|{\int_{x}^{x_0} t^{n-1} H_1(t) f(t)\,dt}|$ goes to $\infty$ because otherwise \eqref{eq:thelimitIwant} follows after passing to a subsequence. Then L'H\^opital's rule requires us to check that
\begin{align}
\frac{\frac{d}{dx}\int_x^{x_0} t^{n-1}H_1(t)f(t)\,dt}{\frac{d}{dx}\frac{H_1(x)}{H_2(x)}} = \frac{x^{n-1}H_1(x)f(x)}{\frac{W(x)}{H_2(x)^2}} \to 0,
\end{align}
and we can easily see that this is true by using \eqref{eq:wronski} and \eqref{eq:bessel1}--\eqref{eq:bessel2}. In sum, if $|f| = O(1)$ as $x \to 0^+$ and if $v$ solving \eqref{eq-u} remains uniformly bounded as well, then
\begin{align}\begin{split}
\label{eq-Sln-u}
v(x)  &= \left( \frac{v(x_0)}{H_2(x_0)} +\frac{2H_1(x_0)}{H_2(x_0)} \int_0^{x_0} t^{n-1} H_2(t)f(t)  \,dt \right) H_2(x)\\
&- 2H_1(x)  \int_0^{x} t^{n-1} H_2(t) f(t)\, dt - 2H_2(x)  \int_x^{x_0} t^{n-1} H_1(t) f(t)\,dt.
\end{split}
\end{align}

To apply this result to our situation, we now let $\lambda = \lambda_\ell$, $f = \mathring{G}_\ell$ and $v = \mathring{u}_\ell$ for all $\ell \geq 1$. Thanks to our assumption that $u = \mathbf{O}(x^M)$ for some $M > 1$, it is obvious that $v$ and $f$ are $O(1)$, so \eqref{eq-Sln-u} holds. In fact, only $M \geq 0$ is needed for this. For the missing case $\ell = 0$ we use the analogous representation formula \eqref{eq-Sln-lambda0}, and here we only require that $M > \frac{1}{2}$. Thus, ignoring convergence issues,
\begin{align}\begin{split}\label{eq-Soln-Bessel}
 \mathring{u}(x) &= \boxed{\left( \frac{\mathring{u}_0(x_0)}{x_0} +\frac{x_0^{-n-2}}{n+2} \int_0^{x_0} t^{n} \mathring{G}_0(t)\,  dt \right) x} \\
&- \frac{x^{-n-2}}{n+2}  \int_0^{x} t^{n} \mathring{G}_0(t)\, dt -  \boxed{\frac{x}{n+2} \int_x^{x_0} t^{-2} \mathring{G}_0(t)\,
dt}\\
& + \sum_{\ell=1}^\infty  \left( \frac{\mathring{u}_{\ell}(x_0)}{H_{2,\ell}(x_0)} +\frac{2H_{1,\ell}(x_0)}{H_{2,\ell}(x_0)} \int_0^{x_0} t^{n-1} H_{2,\ell}(t) \mathring{G}_{\ell}(t)\,  dt \right) H_{2,\ell}(x)\varphi_\ell\\
& - \sum_{\ell=1}^\infty\left( 2H_{1,\ell}(x) \int_0^{x} t^{n-1} H_{2,\ell}(t) \mathring{G}_{\ell}(t)\, dt\right)\varphi_\ell-\sum_{\ell=1}^\infty \left(2H_{2,\ell}(x) \int_x^{x_0}  t^{n-1}H_{1,\ell}(t) \mathring{G}_{\ell}(t)\,
dt \right)\varphi_\ell.
\end{split}
\end{align}
The following claim, which is not optimal but sufficient, makes this rigorous. \medskip\

\noindent \emph{Claim}: The three infinite sums in \eqref{eq-Soln-Bessel} converge in $C^0_{\rm loc}$. Moreover, for all $N > 0$ there exists a small number $x_{0,N}>0$ such that if $x_0 \leq x_{0,N}$, then each of these sums is pointwise $O(x^N)$ as $x \to 0^+$.\medskip\

Assume for a moment that the Claim is true. Note that $|t^n\mathring{G}_0(t)| = O(t^{n + 2M})$ because $u = \mathbf{O}(x^M)$. Thus, the first un-boxed term on the right-hand side of \eqref{eq-Soln-Bessel} is $O(x^{2M-1}) = o(x)$ because $M>1$. By writing $\int_{x}^{x_0} = \int_0^{x_0} - \int_0^x$, we can split the second box in \eqref{eq-Soln-Bessel} into a constant multiple of $x$ plus a term which is ${O}(x^{2M}) = o(x)$ because $M > \frac{1}{2}$. Then the constant multiple of $x$ must cancel the first box in \eqref{eq-Soln-Bessel} to make $\mathring{u}(x) = o(x)$, so the proposition follows. Thus, it remains to prove the Claim.

We begin with two preliminary estimates. First, there is a constant $C$ such that for all $\ell \geq 1$,
\begin{align}\label{eq:auxspec1}
\|\varphi_\ell\|_{L^\infty(D)}\leq C \ell^{\frac{1}{2}}.
\end{align}
This follows from a standard estimate for eigenfunctions of the Laplacian on a closed manifold together with Weyl's law. Second, for all $x \in (0,1]$, all $\ell \geq 1$ and all $N_1,N_2>0$,
\begin{align}\label{eq:auxspec2}
|\mathring{u}_\ell(x)| + |\mathring{G}_\ell(x)| \leq C_{N_1,N_2} \ell^{-N_1}x^{N_2}.
\end{align}
Indeed, for any fixed $x$, integration by parts over $D$ shows as usual that $\mathring{u}_\ell(x), \mathring{G}_\ell(x)$ decay faster than any power of $\ell$, and the constants in these estimates can be bounded by the norms of sufficiently high tangential derivatives of $\mathring{u}(x), \mathring{G}(x)$. Since $u = \mathbf{O}(x^M)$, any tangential derivative of $\mathring{u}(x)$ is $O(x^\infty)$. The same holds for $\mathring{G}(x)$ because as in the proof of Lemma \ref{lem-UF}, we can decompose $u$ into an $\mathbf{O}(x^M)$ radial part plus $\mathbf{O}(x^\infty)$, and this shows that $\mathring{G}(x)$ decomposes into an $\mathbf{O}(x^{2M})$ radial part plus $\mathbf{O}(x^\infty)$.

We now apply \eqref{eq:auxspec1}--\eqref{eq:auxspec2} to each term of the three infinite sums in \eqref{eq-Soln-Bessel}. Then we can see that in order to prove the Claim, it will be enough to prove a bound of the form
\begin{align}\label{eq:quantities}
\frac{H_{2,\ell}(x)}{H_{2,\ell}(x_0)} + H_{1,\ell}(x)\sup_{0<t\leq x} \biggl(t^{\bar{N}}H_{2,\ell}(t)\biggr) + H_{2,\ell}(x) \sup_{x \leq t \leq x_0} \biggl(t^{\bar{N}} H_{1,\ell}(t)\biggr) \leq C_{x_0,\bar{N}}\ell^{N_0}x^{\bar{N}-N_0}
\end{align}
for some universal $N_0>0$ and for all $\bar{N} \geq N_0$. This bound needs to hold for all $x \in (0,x_0]$ and $\ell \geq 1$, but we can assume that $x_0$ is less than some small number $x_{0,\bar{N}}>0$ depending only on $\bar{N}$.

By the remark after \eqref{eq:bessel1}--\eqref{eq:bessel2}, for all $\epsilon >0$ there exists an $x_{0,\epsilon}>0$ independent of $\ell$ such that for all $0 < x \leq x_0 \leq x_{0,\epsilon}$, all $\ell \geq 1$ and all $\bar{N}>0$, the left-hand side of \eqref{eq:quantities} is bounded by
\begin{align}\label{eq:quantities1}
\begin{split}
&(1+\epsilon)\left(\frac{x}{x_0}\right)^{-\frac{n}{2}+\frac{1}{4}}x^{\bar{N}} \sup_{0 < t \leq x_0} \left(t^{-\bar{N}}e^{2\sqrt{\lambda_\ell}\left(\frac{1}{\sqrt{x_0}}-\frac{1}{\sqrt{t}}\right)}\right)^{\begin{small}\textcircled{1}\end{small}}\\
&+\frac{1+\epsilon}{4\sqrt{\lambda_\ell}}x^{-\frac{n}{2}+\frac{1}{4}}\left[e^{\frac{2\sqrt{\lambda_\ell}}{\sqrt{x}}}\sup_{0<t\leq x} \left(t^{-\frac{n}{2}+\frac{1}{4}+\bar{N}}e^{-\frac{2\sqrt{\lambda_\ell}}{\sqrt{t}}}\right)^{\begin{small}\textcircled{2}\end{small}} + e^{-\frac{2\sqrt{\lambda_\ell}}{\sqrt{x}}}\sup_{x \leq t \leq x_0}\left(t^{-\frac{n}{2}+\frac{1}{4}+\bar{N}}e^{\frac{2\sqrt{\lambda_\ell}}{\sqrt{t}}}\right)^{\begin{small}\textcircled{3}\end{small}}\right].
\end{split}
\end{align}
As a function of $t$, $\textcircled{1}$ is increasing from $t = 0$ until it reaches its global maximum. Assume $\bar{N} \geq \frac{n}{2}-\frac{1}{4}$. Then $\textcircled{2}$ is increasing for all $t > 0$, and $\textcircled{3}$ is decreasing from $t = 0$ until it reaches its global minimum. The maximum of $\textcircled{1}$ and the minimum of $\textcircled{3}$ are obviously attained for $t/{\lambda_\ell}$ equal to some constant depending only on $\bar{N}$, i.e., these functions are monotone on $(0,x_0]$ if $x_0 \leq x_{0,\bar{N}} \ll 1$. Thus,
\begin{align}
\eqref{eq:quantities1} \leq C_{x_0,\bar{N}}x^{\bar{N} - \frac{n}{2}+\frac{1}{4}} + C\lambda_\ell^{-\frac{1}{2}}x^{\bar{N}-n+\frac{1}{2}}
\end{align}
for all $\bar{N} \geq \frac{n}{2}-\frac{1}{4}$, all $0 < x \leq x_0 \leq x_{0,\bar{N}}$ and all $\ell \geq 1$, where $C$ is a universal constant independent of all of these parameters. This implies \eqref{eq:quantities} (with a large safety margin), hence the Claim after \eqref{eq-Soln-Bessel}, and hence, as we have already explained, the proposition.
\end{proof}

To make the above proposition useful for us, we also need the following calculus lemma. This is a sharper version of an estimate that was already used during the proof of the proposition.

 \begin{lemma}\label{lem-Hil-xk}
Let $c > 0$ and $k> -\frac{3}{2}$. Then for all $x > 0$ we have that
 \begin{align}\label{eq:firstintegral}
\int_{0}^{x} t^k e^{ -\frac{c}{\sqrt{t}}} \,dt < \frac{2}{c} x^{k+\frac{3}{2}} e^{- \frac{c}{\sqrt{x}}}.
\end{align}
Moreover,  for all $\varepsilon>0$ and for all $0<x\leq x_0$ with $x_0 \leq (\frac{\epsilon}{2+\epsilon} \frac{c}{2k+3})^2$ we have that
\begin{align}\label{eq:secondintegral}
\int_{x}^{x_0} t^k e^{\frac{c}{\sqrt{t}}} \,dt < \frac{2+\epsilon}{c} x^{k+\frac{3}{2}} e^{\frac{c}{\sqrt{x}}}.
\end{align}
\end{lemma}

\begin{proof}
For \eqref{eq:firstintegral} we begin by observing that for all $c > 0$ and all $k > -\frac{3}{2}$,
\begin{align}\label{eq:1stint-bdry}
\lim_{x\rightarrow 0}\frac{c \int_0^{x} t^k e^{-\frac{c}{\sqrt{t} }}\, dt}{x^{k+\frac{3}{2}} e^{-\frac{c}{\sqrt{x}}}}& = 2,\;\,
\lim_{x\rightarrow \infty}\frac{c \int_0^{x} t^k e^{-\frac{c}{\sqrt{t} }}\, dt}{x^{k+\frac{3}{2}} e^{-\frac{c}{\sqrt{x}}}} = 0.
\end{align}
This is an application of L'H\^opital's rule. The limit as $x \to 0$ is always of the form $\frac{0}{0}$, so L'H\^opital can be applied directly, while for $x \to \infty$ we need to consider the two cases $-\frac{3}{2} < k < -1$, where the limit is of the form $\frac{\text{finite}}{\infty} = 0$, and $k \geq -1$, where it takes the form $\frac{\infty}{\infty}$ and L'H\^opital can be used. 

If the function whose limits we found in \eqref{eq:1stint-bdry} has no critical points, we are done. If $x \in (0,\infty)$ is a critical point, then the critical point equation implies that $c + (2k+3)\sqrt{x} > 0$ and
\begin{align}
\frac{c \int_0^{x} t^k e^{-\frac{c}{\sqrt{t} }}\, dt}{x^{k+\frac{3}{2}} e^{-\frac{c}{\sqrt{x}}}} = \frac{2c}{c + (2k+3)\sqrt{x}}.
\end{align} 
Because $k > -\frac{3}{2}$, this ratio is strictly less than $2$, so the function cannot attain a global maximum at $x$ and is therefore strictly less than $2$ everywhere in this case as well.

The argument for \eqref{eq:secondintegral} is similar. The limits of the corresponding function as $x \to 0$ and $x \to x_0$ are $2$ and $0$, respectively. If $x \in (0,x_0)$ is a critical point, then $c - (2k+3)\sqrt{x} > 0$ and
\begin{align}
\frac{c \int_x^{x_0} t^k e^{\frac{c}{\sqrt{t} }}\, dt}{x^{k+\frac{3}{2}} e^{\frac{c}{\sqrt{x}}}} = \frac{2c}{ c- \left(2 k+3\right) \sqrt{x}}.
\end{align}
This ratio is strictly greater than $2$, but if $x_0 \leq (\frac{\epsilon}{2+\epsilon}\frac{c}{2k+3})^2$ then it is strictly less than $2+\epsilon$.
\end{proof}

\subsection{The continuity argument}

Fix $N \in (0,1)$ once and for all. Define
\begin{equation}
H(x) := x^{ - \frac{n}{2} + \frac{1}{4}}e^{-\frac{2\sqrt{\lambda_1}}{\sqrt{x}}}.
\end{equation}
Throughout this section, write $v(t)$ to denote the restriction of a function $v$ on the total space to the compact slice $\{x=t\}$. Given $x_0 \in (0,1]$, $C>0$ and $x \in (0,x_0]$, consider the statement
\begin{equation}\label{genesis}
\mathbf{S}_{x_0,C}(x): \text{``For all $y \in (0,x_0]$, we have that}\,\max |u(y)| \leq \begin{cases}CH(y)&\text{if}\;\,y \in [x,x_0],\\
CH(x)(\frac{y}{x})^N&\text{if}\;\,y\in(0,x].\text{''}\end{cases}
\end{equation}

\begin{proposition}\label{prop:key-prop}
Assume only that $u = \mathbf{O}(x^{M})$ as $x \to 0^+$ for some $M > \max\{\frac{3}{2},2-N\}$ rather than $|u| = O(x^\infty)$. There exist $x_0 \in (0,1]$ and $C >0$ such that
$\inf\{x \in (0,x_0]: \mathbf{S}_{x_0,C}(x)\;\text{is true}\} = 0$.
\end{proposition}

If this is true, then $u$ satisfies the pointwise decay estimate stated in Reduction \ref{thm-main-red} uniformly on all intervals $[x_i,x_0]$ for some infinite sequence $x_i \to 0^+$. Thus, the Main Theorem is proved.

For the proof of Proposition \ref{prop:key-prop} we need the following barrier property of the function $x^N$. This is a standard application of the maximum principle except for one twist.

\begin{lemma}\label{lem:barrier}
Assume only that $u = \mathbf{O}(x^M)$ as $x \to 0^+$ for some $M > \frac{N}{2}$. There exists an $x_1 > 0$ such that the following is true. Let $\Omega \subset \{x < x_1\}$ be an open set. Let $p\in\Omega$. Let $c > 0$. It is not possible to have $u \leq c x^N$ in $\Omega$ with equality at $p$, or to have $u \geq -c x^N$ in $\Omega$ with equality at $p$.
\end{lemma}

\begin{proof}
Define $B = c x^N$. If $u \leq B$ in $\Omega$ with equality at $p$, then $(\ddbar B)(p) \geq (\ddbar u)(p)$, so in particular $(\omega_h + \ddbar B)(p) > 0$ and the Monge-Amp\`ere operator $M_h(B)$ as in \eqref{eq:MAOp} is defined at $p$. In addition, $0 = M_h(u)\leq M_h(B) \leq L_h(B) = -\gamma B < 0$ at $p$, where $\gamma = (1-N)(n+1+N) > 0$ and where we have used the concavity of $\log \det$ on the space of Hermitian matrices. This is a contradiction.

Fix $C_M$ so that $|\ddbar u|_{\omega_h} \leq C_M x^{M}$. Note that $|\ddbar x^N|_{\omega_h} = \delta x^N$ for some universal $\delta>0$. There is also a universal $\epsilon>0$ such that $|Q_h(v)| \leq \epsilon|\ddbar v|$\begin{footnotesize}$^2_{\omega_h}$\end{footnotesize} for all functions $v$ at all points where $|\ddbar v|$\begin{footnotesize}$_{\omega_h}$\end{footnotesize} $\leq 1$. 
Then choose $x_1$ so that $\max\{C_M x_1^M,\delta x_1^N\} \leq 1$ and $\epsilon C_M^2 x_1^{2M-N} \leq \lambda\gamma$, where $\lambda = \min\{1,\frac{\gamma}{\delta\epsilon}\}$.

Now suppose that $u \geq -B$ in $\Omega$ with equality at $p$. If $c < \lambda$, then in particular $c < 1$, so we have $|\ddbar B|_{\omega_h} = c\delta x^N < 1 <\sqrt{n}= |\omega_h|_{\omega_h}$, hence $\omega_h - \ddbar B > 0$, on $\Omega$. Thus, the Monge-Amp\`ere operator $M_h(-B)$ is defined on $\Omega$, and $0 = M_h(u) \geq M_h(-B) = \gamma B + Q_h(-B) \geq (\gamma c - \epsilon c^2 \delta^2 x^N) x^N > 0$ at $p$ because $\delta x^N < 1$ and $c < \frac{\gamma}{\delta\epsilon}$. This is a contradiction. We also have $-Q_h(u) = L_h(u) \geq L_h(-B) = \gamma B$ at $p$. Thus, if $c \geq \lambda$, then $\lambda \gamma x^N \leq |Q_h(u)| \leq \epsilon C_M^2 x^{2M}$ at $p$. This is another contradiction.
\end{proof}

\begin{proof}[Proof of Proposition \ref{prop:key-prop}]
We begin by setting notation and explaining some conventions that will be in force throughout the proof. Fix $M > \max\{\frac{3}{2},2-N\}$ and a constant $C_M$ such that $\max |u(x)| \leq C_M x^{M}$ for all $x \in (0,1]$. We let $C_0$ denote a generic constant that can increase from line to line or even within the same line but that is not allowed to depend on the constants $x_0,C$, which are still to be chosen, or on any unproved properties of $u$. By contrast, $C$ will always denote the specific constant in \eqref{genesis}. 

We will show that the proposition holds for all sufficiently small values of $x_0$ and for all sufficiently large values of $C$. Thus, there is no harm in assuming that $C \geq 1$ and $x_0 \leq \frac{1}{2}$. We will also assume that $x_0$ is so small that $x_0\leq x_1$ as in Lemma \ref{lem:barrier}, that the representation formula \eqref{eq-Soln-Bessel-3} holds on $(0,x_0]$ $\times$ $D$, and that $\max |u(x) - \mathring{u}(x)| \leq \frac{1}{6} H(x)$ for all $x \in (0,x_0]$. The latter is possible by Lemma \ref{lem:circle}.

Consider any two numbers $x_0$ and $C$ satisfying these constraints. We begin by observing that if, in addition, $C H(x_0) \geq C_M x_0^N$, then the statement $\mathbf{S}_{x_0,C}(x_0)$ is true because $\max |u(x)| \leq C_M x^M$ for all $x \in (0,1]$, where $M \geq N$. Thus, the set of all $x \in (0,x_0]$ for which $\mathbf{S}_{x_0,C}(x)$ is true is not empty.

The infimum, $x_*$, of this set is an element of $[0,x_0]$. We may assume that $x_* > 0$ for all admissible choices of $x_0$ and $C$, aiming to derive a contradiction if $x_0$ is small and $C$ is large. Because $\max |u(y)|$ is continuous in $y$, clearly $\mathbf{S}_{x_0,C}(x)$ is true for $x = x_*$ but not for any $x < x_*$.\medskip\

\noindent \emph{Claim 1}: There exists a $y = y_* \in (0,x_*]$ for which \eqref{genesis} with $x = x_*$ is an equality.\medskip\

\noindent \emph{Proof}: If the inequality \eqref{genesis} with $x = x_*$ was strict for all $y \in (0,x_*]$, then, again because $\max |u(y)|$ is continuous in $y$, the first line of \eqref{genesis} would obviously still be true for all $x \in [x_*-\epsilon,x_*)$ for some $\epsilon \in (0,\frac{x_*}{2}]$. We will now show that the second line of \eqref{genesis} also still holds for all $x \in [x_*-\epsilon,x_*)$ after decreasing $\epsilon$ if needed. If so, then $\mathbf{S}_{x_0,C}(x)$ holds for all $x \in [x_*-\epsilon,x_*)$, contradicting the minimality of $x_*$ and proving the Claim. Thus, we need to prove that $\max |u(y)| \leq CH(x)(\frac{y}{x})^N$ for all $y \in (0,x]$ and all $x \in [x_*-\epsilon,x_*)$, given that this is true for $x = x_*$ with strict inequality for all $y \in (0,x_*]$. 

To do so, we begin by defining $\bar{x}_*$ to be the largest number in $(0,x_*]$ such that
\begin{align}\label{eq:xbarstar}
C_M \bar{x}_*^{M-N} \leq \min\{x^{-N} H(x):x_*/2\leq x \leq x_*\}.
\end{align}
This exists because $M > N$. Then for all $y \in (0,\bar{x}_*)$ it directly follows from \eqref{eq:xbarstar}, $C \geq 1$ and $\epsilon\leq\frac{x_*}{2}$ that $\max |u(y)| \leq C_M y^M < CH(x)(\frac{y}{x})^N$ for all $x \in [x_*-\epsilon,x_*)$. Thus, it remains to prove the desired inequality for all $y \in [\bar{x}_*,x]$ and all $x \in [x_*-\epsilon,x_*)$. Then simply note that after decreasing $\epsilon$ if needed the inequality actually holds for all $y \in [\bar{x}_*,x_*]$ because on the compact interval $[\bar{x}_*,x_*]$ the continuous function $y \mapsto \max |u(y)|$ is pointwise strictly smaller than the continuous function $y \mapsto CH(x_*)($\begin{footnotesize}$\frac{y}{x_*}$\end{footnotesize}$)^N$, which is uniformly approximated by the functions $y \mapsto CH(x)(\frac{y}{x})^N$ as $x \to x_*$.\hfill $\Box_{\text{Claim 1}}$\medskip\

Lemma \ref{lem:barrier} rules out the possibility that $y_* < x_*$, so $y_* = x_*$ and hence $\max |u(x_*)| =C H(x_*) > 0$. For a contradiction it then clearly suffices to prove that $\max |u(x_*)| \leq \frac{5}{6} C H(x_*)$ if $x_0$ is small and $C$ is large. This follows if we can prove that $\max |\mathring{u}(x_*)| \leq \frac{2}{3}C H(x_*)$, and doing so will take up the rest of this paper. The idea is to bound $\max |\mathring{u}(x_*)|$ by applying the representation formula \eqref{eq-Soln-Bessel-3} at $x = x_*$, whose right-hand side is a linear functional of $\mathring{G}$ (over the whole end $\{0<x \leq x_0\}$), i.e., a quadratic or higher functional of $\ddbar u$. To make this precise, we first bound $\ddbar u$ in terms of $u$ using Schauder theory (see the following Claim) and then bound $u$ pointwise by using \eqref{genesis} for $x = x_*$.

To state Claim 2 we need some more notation. For $t \in (0,x_0]$ and $r \in (0,\frac{1}{2}]$ define 
\begin{align}
A_r(t) := \{(1-r)t < x < (1+r)t\}.
\end{align}
Recall here that $\rho = -\sqrt{(n+1)/2}\log x$ satisfies $|d\rho|_{\omega_h}=1$. Thus, for all $p \in A_r(t)$ with $x(p) = t$ there exists an $\omega_h$-geodesic ball of radius comparable to $r$ centered at $p$ which is contained in $A_r(t)$.\medskip\

\noindent \emph{Claim 2}: For all $Q \in \N_0$ there exists an absolute constant $C_Q$, and for all $\eta \in (0,M-1)$ there exists an absolute number $x_{0,\eta}>0$, such that if $x_0 \leq x_{0,\eta}$, then for all $t \in (0,x_0]$, all $r \in [\sqrt{t}/2,1/2]$ and all $(\ell,Q) \in  \{(0,0)\} \cup (\N_{\geq 1} \times \N_0)$ we have that 
\begin{align}
|\mathring{G}_{\ell}(t)| & \leq C_Q (1+\lambda_{\ell})^{-Q} \left(\frac{\sqrt{t}}{r}\right)^{2Q} \frac{t^\eta}{r^2}\sup_{A_r(t)} |u|.\label{eq-Schauder+1}
\end{align}

For $r \sim \sqrt{t}$, which is already a useful case, one can prove \eqref{eq-Schauder+1} by passing to a local covering space and reducing to the standard Schauder theory on $\C^n$ even if $D$ is a non-flat Calabi-Yau manifold; see the argument after Reduction \ref{thm-main-red}.  However, we also require \eqref{eq-Schauder+1} for $r \gg \sqrt{t}$, and for this we do need to assume that $D$ is flat. But this is the only point in Section \ref{sec:exp-decay} where the flatness of $D$ is used, and it again seems likely that this can be avoided by using the methods of \cite{HeinTosatti}.\medskip\

\noindent \emph{Proof}: We begin with the case $Q = 0$, where $\ell \in \N_0$ is arbitrary. We first claim that
\begin{align}\label{eq:firststepofmany}
|\mathring{G}_{\ell}(t)| \leq C_0 \|\varphi_\ell\|_{L^1(D)} \max_{2 \leq k \leq n} \max_{\{x=t\}} |\ddbar u|_{\omega_h}^k \leq  C_0 \max_{2\leq k \leq n} \left(\frac{1}{r^2}\sup_{A_r(t)} |u|\right)^k.
\end{align}
Here the first inequality is clear from \eqref{eq-PDE}--\eqref{eq:fourier}. The second inequality follows from Schauder theory applied on balls of radius $\sim r$ in quasi-coordinates for $\omega_h$ (see the proof of Lemma \ref{lem-v_higher}), using also the fact that $\|\varphi_\ell\|_{L^2(D)} = 1$. As $M> 1$, given any $\eta \in (0,M-1)$ we can choose $x_{0,\eta}$ so small that
\begin{equation}
C_M(3/2)^M x_{0,\eta}^{M-1-\eta}\leq 1/4.
\end{equation}
As $\max |u(y)| \leq C_M y^M$, it follows for all $r \in [\sqrt{t}/2, 1/2]$ that
\begin{align}
\sup\nolimits_{A_r(t)} |u| \leq C_M (3t/2)^M \leq t^{\eta}(t/4) \leq t^\eta r^2 < r^2.
\end{align}
This shows that the $k = 2$ term in \eqref{eq:firststepofmany} dominates. It also allows us to estimate this quadratic term by a small linear term (which is actually more convenient), proving \eqref{eq-Schauder+1} for $Q = 0$. 

For $Q \geq 1$ and $\ell \geq 1$, we first integrate by parts $Q$ times in the definition of $\mathring{G}_\ell(t)$, which brings out a factor of $\lambda_\ell^{-Q}$ while replacing $\mathring{G}(t)$ by a $2Q$-th covariant derivative of $\mathring{G}(t)$ with respect to the given flat metric on $D$. We gain $2Q$ factors of $\sqrt{t}$ by replacing each of these covariant derivatives on $D$ by a directional derivative with respect to some quasi-coordinate for $\omega_h$. By Schauder theory, we can trade each of the latter for a factor of $C_0/r$ if we also take the sup over a ball of radius $\sim r$.
\hfill $\Box_{\text{Claim 2}}$\medskip\

Thanks to Claim 2 and \eqref{genesis} for $x = x_*$, we are now able to estimate the right-hand side of \eqref{eq-Soln-Bessel-3} at $x = x_*$ and prove the desired result, i.e., that $\max |\mathring{u}(x_*)| \leq \frac{2}{3}C H(x_*)$ if $x_0$ is small and $C$ is large. As discussed above, this will complete the proof of the proposition and of the Main Theorem.
     
Evaluate \eqref{eq-Soln-Bessel-3} at $x = x_*$, take $L^\infty$ norms over $\{x_*\} \times D$, apply the triangle inequality, and absorb irrelevant numerical constants. This yields the following:
\begin{align}\begin{split}\label{eq-Soln-Bessel-456}
&\frac{1}{2}\max |\mathring{u}(x_*)| \leq
{^{\textcircled{4}} x_*^{-n-1} \int_0^{x_*} t^{n} |\mathring{G}_0(t)|\, dt}+  {^{\textcircled{5}}x_* \int_0^{x_*} t^{-2} |\mathring{G}_0(t)|\,dt}\\
&+ \sum_{\ell=1}^\infty  {^{\textcircled{1}}\left({^{\textcircled{a}}\frac{|u_{\ell}(x_0)|}{H_{2,\ell}(x_0)} +^{\textcircled{b}}\frac{H_{1,\ell}(x_0)}{H_{2,\ell}(x_0)} \int_0^{x_0} t^{n-1} H_{2,\ell}(t) |\mathring{G}_{\ell}(t)|\,  dt}\right) H_{2,\ell}(x_*)\|\varphi_\ell\|_{L^\infty(D)}}\\
&+  \sum_{\ell=1}^\infty{^{\textcircled{3}}\left( H_{1,\ell}(x_*) \int_0^{x_*} t^{n-1}H_{2,\ell}(t) |\mathring{G}_{\ell}(t)|\, dt \right)\|\varphi_\ell\|_{L^\infty(D)}} \\
&+  \sum_{\ell=1}^\infty{^{\textcircled{2}}\left( H_{2,\ell}(x_*) \int_{x_*}^{x_0} t^{n-1} H_{1,\ell}(t)|\mathring{G}_{\ell}(t)|\,
dt \right)\|\varphi_\ell\|_{L^\infty(D)}}.
\end{split}
\end{align}
Here we have labeled all terms in order of increasing conceptual difficulty. We will now estimate these terms one by one. Because there are $6$ contributions in total, it will be sufficient to prove that each of them (after summing over $\ell$ in cases $\textcircled{1}$--$\textcircled{3}$) is bounded above by $C H(x_*)/20$.

$\textcircled{1}$ This is essentially the same as the estimates after \eqref{eq-Soln-Bessel} in the proof of Proposition \ref{prop:rep}. Claim 2 is not needed for this. Indeed, using \eqref{eq:auxspec1}--\eqref{eq:auxspec2} and choosing $x_0$ so small that the remark after \eqref{eq:bessel1} --\eqref{eq:bessel2} can be applied, we obtain as in \eqref{eq:quantities1} that
\begin{equation}
\sum_{\ell=1}^\infty \textcircled{a}
\leq \sum_{\ell=1}^\infty \left[{O}(\ell^{-\infty}x_0^\infty) \left(\frac{x_*}{x_0}\right)^{-\frac{n}{2}+\frac{1}{4}}e^{-2\sqrt{\lambda_{\ell}}\left(\frac{1}{\sqrt{x_*}} - \frac{1}{\sqrt{x_0}}\right)}\right]
\leq O(x_0^\infty) \frac{H(x_*)}{H(x_0)}.
\end{equation}
This will be bounded above by $CH(x_*)/20$, as desired, if we choose $x_0$ to be smaller than some small absolute constant and then choose $C \geq 1/H(x_0)$. Similarly,
\begin{equation}\label{eq-3-2}
\sum_{\ell=1}^\infty \textcircled{b}\leq O(x_0^\infty)\frac{H(x_*)}{H(x_0)} \sum_{\ell=1}^\infty \left[O(\ell^{-\infty}) H_{1,\ell}(x_0)\int_{0}^{x_0} O(t^\infty) H_{2,\ell}(t)\,dt\right] \leq O(x_0^\infty)\frac{H(x_*)}{H(x_0)}.
\end{equation}
This concludes our discussion of $\textcircled{1}$.

In $\textcircled{2}$--$\textcircled{5}$ we will use Claim 2 together with \eqref{genesis}. We begin with some general remarks on the key question of how to choose $r$ in \eqref{eq-Schauder+1}. A smaller $r$ makes \eqref{eq-Schauder+1} weaker. On the other hand, $r$ needs to be chosen sufficiently small so that the function on the right-hand side of \eqref{genesis}, i.e., $y \mapsto CH(y)$ for $y \in [x_*,x_0]$ and $y \mapsto CH(x_*)(\frac{y}{x_*})^N$ for $y \in (0,x_*]$, varies at most by a bounded factor over the interval $y \in ((1-r)t,(1+r)t)$ for all $t$ in the domain of the relevant integral in \eqref{eq-Soln-Bessel-456}. If so, then we are able to bound $\sup_{A_r(t)} |u|$ on the right-hand side of \eqref{eq-Schauder+1} by the value of this function at $y = t$, so we at least retain a chance of estimating the integral by $\sim CH(x_*)$ rather than by some other function of $x_*$ that might decay much more slowly than $H$. For a function of the form $y \mapsto \exp(-\delta y^{-\alpha})$ ($\alpha,\delta>0$) we thus require that $r \lesssim \frac{1}{\delta\alpha} t^\alpha$. We will apply this with $\alpha = \frac{1}{2}$, $\delta$ fixed, and with $\alpha \to 0$, $\delta\alpha$ fixed. However, for us, it will sometimes be beneficial (see $\textcircled{3}$) or at least harmless (see $\textcircled{4}$--$\textcircled{5}$) to choose $r \ll \frac{1}{\delta\alpha}t^\alpha$.

$\textcircled{2}$ Here we have $\alpha = \frac{1}{2}$ because the domain of integration is $t \in [x_*,x_0]$. Thus, our only option is to set $r \sim \sqrt{t}$ in \eqref{eq-Schauder+1}. More precisely, we let $r = \sqrt{t}/2$. Then we claim that, as expected,
\begin{equation}\label{eq:victory111}
\sup\nolimits_{A_r(t)} |u| \leq C_0 C H(t)\;\,\text{for all}\;\,t \in [x_*,x_0].
\end{equation}
Indeed, if $A_r(t) \subset \{x_* \leq x \leq x_0\}$, this follows from the first line of \eqref{genesis} (with $x = x_*$) precisely by the above argument motivating the choice $r \lesssim \sqrt{t}$. For $t$ so close to $x_*$ that $A_r(t) \not\subset \{x_* \leq x \leq x_0\}$, we also need to use the second line of \eqref{genesis} and the trivial fact that $\sup_{A_r(t) \cap \{x < x_*\}} CH(x_*)(\frac{x}{x_*})^N = CH(x_*)$. Similary, for $t$ so close $x_0$ that $A_r(t) \not\subset \{x_* \leq x \leq x_0\}$, we moreover need to use our assumption that $|u|  \leq C_M x^M$ on $A_r(t) \cap \{x_0 < x \leq 1\}$ and choose $C$ large enough so that
\begin{equation}\label{eq:choiceofC}
C_M(x_0(1+\sqrt{x_0}/2))^M \leq C H(x_0).
\end{equation}
Having proved \eqref{eq:victory111}, we now combine this with \eqref{eq-Schauder+1} to obtain that
\begin{equation}
\begin{split}
\sum_{\ell=1}^\infty \textcircled{2}
&\leq \sum_{\ell=1}^\infty O(\ell^{-\infty})\, x_*^{-\frac{n}{2}+\frac{1}{4}}e^{-\frac{2\sqrt{\lambda_{\ell}}}{\sqrt{x_*}}}\int_{x_*}^{x_0} t^{n-1} t^{-\frac{n}{2} + \frac{1}{4}}e^{\frac{2\sqrt{\lambda_{\ell}}}{\sqrt{t}}} t^{\eta-1} CH(t) \, dt\\
&\leq \sum_{\ell=1}^\infty O(\ell^{-\infty})\, x_*^{-\frac{n}{2}+\frac{1}{4}}e^{-\frac{2\sqrt{\lambda_{\ell}}}{\sqrt{x_*}}}\int_{x_*}^{x_0} C t^{-\frac{3}{2}+\eta} e^{\frac{2\sqrt{\lambda_{\ell}} -2\sqrt{\lambda_1}}{\sqrt{t}}}dt.
\end{split}
\end{equation}
This holds for any given $\eta \in (0,M-1)$ if $x_0 \leq x_{0,\eta}$ for some absolute number $x_{0,\eta}>0$. 

We now choose $\eta > \frac{1}{2}$ (this is possible because $M > \frac{3}{2}$), evaluate the integral directly if $\lambda_\ell = \lambda_1$, and estimate it using Lemma \ref{lem-Hil-xk} if $\lambda_\ell > \lambda_1$. Fixing $\epsilon=1$ in Lemma \ref{lem-Hil-xk} and also assuming that
\begin{equation}
3\eta\sqrt{x_0} \leq \min\{\sqrt{\lambda_\ell}: \ell \geq 2, \lambda_\ell > \lambda_1\}-\sqrt{\lambda_1},\end{equation}
we thus obtain the following estimate:
\begin{equation}
\sum_{\ell=1}^\infty \textcircled{2}\leq C\sum_{\ell=1}^\infty \left[O(\ell^{-\infty})\,x_*^{-\frac{n}{2}+\frac{1}{4}}e^{-\frac{2\sqrt{\lambda_{\ell}}}{\sqrt{x_*}}}\begin{cases}x_0^{\eta-\frac{1}{2}}/(2\eta-1) &\text{if}\;\,\lambda_\ell=\lambda_1\\ x_*^\eta e^{\frac{2\sqrt{\lambda_\ell}-2\sqrt{\lambda_1}}{\sqrt{x_*}}} &\text{if}\;\,\lambda_\ell > \lambda_1\end{cases}\right]\leq \frac{C_0x_0^{\eta-\frac{1}{2}}\cdot C H(x_*)}{\min\{1,2\eta-1\}} .
\end{equation}
This will be bounded by $CH(x_*)/20$ if we decrease the absolute number $x_{0,\eta}>0$.

{$\textcircled{3}$} While shorter than $\textcircled{2}$, this step may be more confusing because the domain is now $t \in (0,x_*]$, so \eqref{genesis} bounds $|u|$ in terms of a polynomial rather than an exponential, but we are also still integrating this against the exponential-type function $H_{2,\ell}(t)$. Thus, it may not be obvious how $r$ should be chosen (from $r \sim \sqrt{t}$ for exponentials to $r \sim 1$ for polynomials). However, we will see that the ``easy'' choice, $r \sim\sqrt{t}$, which does not require us to assume that $D$ is flat, is actually better and sufficient.

First, while the analog of \eqref{eq:victory111} would still hold for $A_r(t) \subset \{x \leq x_*\}$ with $r \sim 1$, it would be false for $A_r(t) \not\subset \{x \leq x_*\}$ with $r \sim 1$ because the barrier function of \eqref{genesis} is exponential on $(x_*,x_0]$. It is however true in both cases with $r \sim \sqrt{t}$. Thus, setting $r = \sqrt{t}/2$,
\begin{equation}\label{eq:victory112}
\sup_{A_r(t)} |u| \leq C_0 C H(x_*)\left(\frac{t}{x_*}\right)^N\text{for all}\;\,t \in (0,x_*],
\end{equation}
where the case $A_r(t) \not\subset (0,x_*]$ follows from the fact that
\begin{equation}\label{eq:painfullytrivial}
t(1+\sqrt{t}/2) > x_* \Longrightarrow \sup\{ H(y): y \in (x_*, t(1+\sqrt{t}/2)]\}   \leq C_0  H(x_*)\leq C_0 H(x_*) \left(\frac{t}{x_*}\right)^N.
\end{equation}
Second, by combining \eqref{eq-Schauder+1} and \eqref{eq:victory112}, we obtain for all $\eta \in (0,M-1)$ that if $x_0 \leq x_{0,\eta}$, then
\begin{equation}\label{eq-huha}
\begin{split}
\sum_{\ell=1}^\infty \textcircled{3}
&\leq \sum_{\ell=1}^\infty O(\ell^{-\infty}) \, x_*^{-\frac{n}{2}+\frac{1}{4}}e^{\frac{2\sqrt{\lambda_{\ell}}}{\sqrt{x_*}}}\int_0^{x_*} t^{n-1} t^{-\frac{n}{2}+\frac{1}{4}}e^{-\frac{2\sqrt{\lambda_{\ell}}}{\sqrt{t}}} t^{\eta-1}C H(x_*)\left(\frac{t}{x_*}\right)^N dt\\
&\leq \sum_{\ell=1}^\infty O(\ell^{-\infty}) \, x_*^{-\frac{n}{2}+\frac{1}{4}-N}e^{\frac{2\sqrt{\lambda_\ell}}{\sqrt{x_*}}} \int_0^{x_*} t^{\frac{n}{2}-\frac{7}{4}+\eta + N} e^{-\frac{2\sqrt{\lambda_{\ell}}}{\sqrt{t}}} dt \cdot CH(x_*).
\end{split}
\end{equation}
We can now apply Lemma \ref{lem-Hil-xk} directly and without assuming that $\eta > \frac{1}{2}$:
\begin{equation}
\frac{1}{CH(x_*)}\sum_{\ell=1}^\infty \textcircled{3} \leq \sum_{\ell=1}^\infty O(\ell^{-\infty}) \, x_*^{-\frac{n}{2}+\frac{1}{4}-N}e^{\frac{2\sqrt{\lambda_\ell}}{\sqrt{x_*}}} x_*^{\frac{n}{2}-\frac{1}{4}+\eta+N} e^{-\frac{2\sqrt{\lambda_{\ell}}}{\sqrt{x_
*}}} \leq C_0x_*^\eta.
\end{equation}

{$\textcircled{4}$} Here and in $\textcircled{5}$, the upper barrier for $|u|$ from \eqref{genesis} and the function against which we need to integrate this barrier over the interval $(0,x_*]$ in \eqref{eq-Soln-Bessel-456} are polynomials. Away from $t = x_*$ this suggests that we should apply Claim 2 with $r \sim 1$. It is actually possible to work with $r \sim t^\alpha$ for any $\alpha \in [0,\frac{1}{2})$, but choosing $\alpha > 0$ makes the algebra longer while in the proof of Claim 2 only the Schauder estimates for $\alpha = \frac{1}{2}$ are easy and the ones for $\alpha \in [0,\frac{1}{2})$ are all equally difficult, requiring us to assume that $D$ is flat. Thus, for simplicity, we will stick to the choice $r \sim 1$. However, as $t$ approaches $x_*$ we do need to choose $r$ smaller than this because the right-hand side of \eqref{genesis} is exponential rather than polynomial on $(x_*,x_0]$. As it turns out, the following straightforward choice will be good enough:
\begin{equation}\label{eq-radius-choice}
r = r(t) := \max\left\{\frac{1}{2}\left(1 - \frac{t}{x_*}\right), \frac{1}{2}x_*^{\frac{1}{2}}\right\}\,\text{for all}\;\,t \in (0,x_*].
\end{equation}

To see that this makes \eqref{eq:victory112} true for all $t \in (0,x_*]$, first observe that \eqref{genesis} trivially implies \eqref{eq:victory112} if the $\sup$ is taken over $A_r(t) \cap \{x \leq x_*\}$. Now $A_r(t) \not\subset \{x \leq x_*\}$ is possible only if $t \in (x$\begin{footnotesize}$_*$\end{footnotesize}$-\,x$\begin{footnotesize}$_*^{3/2}$\end{footnotesize}$,x_*]$ as otherwise the max in \eqref{eq-radius-choice} is realized by the first term and then $t(1 + r) \leq x_*$ by computation. But if $t \in (x$\begin{footnotesize}$_*$\end{footnotesize}$-\,x$\begin{footnotesize}$_*^{3/2}$\end{footnotesize}$,x_*]$, then $r \leq \sqrt{t}$, so we can deal with the sup over $A_r(t) \cap \{x_* < x\}$ (if nonempty) by replacing $\sqrt{t}/2$ with $\sqrt{t}$ in \eqref{eq:painfullytrivial}, which makes no difference.

As in $\textcircled{3}$, given \eqref{eq:victory112}, we can proceed using \eqref{eq-Schauder+1}: for all $\eta \in (\frac{1}{2},M-1)$, if $x_0 \leq x_{0,\eta}$, then
\begin{equation}\label{eq-Genius}
\frac{\textcircled{4}}{CH(x_*)} \leq x_*^{-n-1}\int_0^{x_*} t^n \left[\frac{t^\eta}{r(t)^2}\left(\frac{t}{x_*}\right)^N \right] dt \leq x_*^{\eta-1}\int_0^{x_*}\frac{dt}{r(t)^2}.
\end{equation}
Here we have simply estimated all powers of $t$ by powers of $x_*$. 
We split the integral in \eqref{eq-Genius} into two parts: from $0$ to $x$\begin{footnotesize}$_*$\end{footnotesize}$-\,x$\begin{footnotesize}$_*^{3/2}$\end{footnotesize} (part I) and from $x$\begin{footnotesize}$_*$\end{footnotesize}$-\,x$\begin{footnotesize}$_*^{3/2}$\end{footnotesize} to $x_*$ (part II). Then
\begin{equation}
\eqref{eq-Genius}_{\rm I}\leq 4x_*^{\eta-1}\int_0^{x_* - x_*^{3/2}} \left(1-\frac{t}{x_*}\right)^{-2} dt \leq 4x_*^{\eta-\frac{1}{2}}, \;\,
\eqref{eq-Genius}_{\rm II} \leq 4x_*^{\eta-1} x_*^{\frac{3}{2}}x_*^{-1} = 4x_*^{\eta-\frac{1}{2}}.
\end{equation}

{$\textcircled{5}$} We can proceed as in $\textcircled{4}$ except that now not all powers of $t$ are positive. Thus, even though we know that the integral converges because $|\mathring{G}_0(t)| = O(t^M)$ with $M > 1$, in order to prove an effective inequality via \eqref{eq-Schauder+1} and \eqref{eq:victory112} we need to impose conditions on these powers of $t$ so that all integrals along the way have an acceptable finite bound. This is where the condition $M > \max\{\frac{3}{2},2 - N\}$ comes in: it lets us choose an $\eta \in (\frac{1}{2},M-1)$ such that $-2+\eta+N>-1$. 
With this, and avoiding bounding powers of $t$ by powers of $x_*$ if the exponents are negative, we obtain for $x_0 \leq x_{0,\eta}$ that
\begin{equation}\label{eq-Genius2}
\begin{split}
\frac{\textcircled{5}}{CH(x_*)} \leq x_*^{1-N} \int_0^{x_*} t^{-2+\eta+N} \frac{dt}{r(t)^2}.
\end{split}
\end{equation}
If $-2+\eta+N\geq 0$, we can finish the proof exactly as in $\textcircled{4}$. Otherwise we split this integral into three parts: from $0$ to $x_*/2$ (part I), where $r(t) \sim 1$, so we can drop $r(t)$ and integrate out the $t^{-2+\eta+N}$ term; and from $x_*/2$ to $x$\begin{footnotesize}$_*$\end{footnotesize}$-\,x$\begin{footnotesize}$_*^{3/2}$\end{footnotesize} (part II) and from $x$\begin{footnotesize}$_*$\end{footnotesize}$-\,x$\begin{footnotesize}$_*^{3/2}$\end{footnotesize} to $x_*$ (part III), where $t^k\sim x_*^k$ for all $k \in \R$, so that we only need to integrate or estimate the $r(t)$ term as in {$\textcircled{4}$}. Parts II and III are easily seen to be bounded on the same order as in $\textcircled{4}$, and part I is even smaller:
\begin{equation}
\begin{split}
\eqref{eq-Genius2}_{\rm I} \leq 16 x_*^{1-N}\int_0^{\frac{x_*}{2}}  t^{-2+\eta+N} \,dt \leq  \frac{16(\frac{1}{2})^{\eta+N-1}}{\eta+N-1} x_*^\eta.
\end{split}
\end{equation}
This completes the proof of the proposition, and hence of the Main Theorem.
\end{proof}

To conclude this paper, we comment on a different geometric problem where the same method can be used to prove sharp estimates of solutions.  This concerns the complete Calabi-Yau metrics of \cite[Thm 4.2]{TY2}, which also play a role in the recent papers \cite{BiqGue,SZ} mentioned in the Introduction.

Interestingly, in the Calabi-Yau setting, scaling the Hermitian metric on the underlying line bundle changes the Calabi model potential by terms that are not $O(1)$, and there are no polynomially decaying solutions or supersolutions to the relevant PDE. This means in particular that barrier arguments and expansion methods as in \cite{JiangShi,LM,RZ,Schu,Wu} cannot be used. However, it was proved in \cite[Cor 1.2]{KK} via separation of variables that bounded solutions converge to constants at infinity faster than any power of the metric distance $\rho$. Independently, using an energy decay argument, it was proved in \cite[p.366]{H} that this convergence rate is $O(\exp(-\delta \rho^{\frac{n}{n+1}}))$ for some $\delta>0$ in dimension $n$.

In the present paper, all of the above methods (barriers, energy decay, separation of variables) were combined together and fed into a new continuity argument to obtain a very precise description of the solutions in the cuspidal K\"ahler-Einstein setting. It is then also possible to go back to the Calabi-Yau setting and use a similar approach (energy decay, separation of variables, continuity) to get a sharper estimate of the convergence rate there. In fact, one can now prove that this rate is
\begin{align}
O(\rho^{-\frac{n-2}{2(n+1)}}\exp(- \delta(n)\sqrt{\lambda_1}\rho^\frac{n}{n+1})),
\end{align}
with $\lambda_1$ the first eigenvalue of the Calabi-Yau metric on the divisor, as conjectured in \cite[p.366]{H}.

\end{document}